\documentclass[11pt]{article}

\usepackage[a4paper,left=20mm,top=20mm,right=20mm,bottom=25mm]{geometry}

\usepackage{amsmath, amssymb, amsfonts, amsfonts, amsthm,latexsym}
\DeclareSymbolFont{yhlargesymbols}{OMX}{yhex}{m}{n}
\DeclareMathAccent{\overarc}{\mathord}{yhlargesymbols}{"F3}

\usepackage{mathtools}
\usepackage[mathscr]{euscript}
\usepackage{wasysym}

\usepackage{algorithm}
\usepackage[noend]{algpseudocode}
\algrenewcommand\algorithmicrequire{\textbf{Input:}}
\algrenewcommand\algorithmicensure{\textbf{Output:}}

\usepackage{graphics}
\usepackage{enumerate}
\usepackage{enumitem}
\usepackage[usenames]{color}
\usepackage[normalem]{ulem}

\usepackage{graphicx}
\usepackage{xcolor}
\usepackage{authblk}

\usepackage[numbers,sort]{natbib}

\usepackage[colorlinks=true,urlcolor=black,linkcolor=blue,citecolor=blue]{hyperref}
\usepackage[font=footnotesize,width=.85\textwidth,labelfont=bf]{caption}

\newtheorem{theorem}{Theorem}[section]
\newtheorem{lemma}[theorem]{Lemma}
\newtheorem{corollary}[theorem]{Corollary}
\newtheorem{definition}{Definition}[section]
\newtheorem{proposition}[theorem]{Proposition}
\newtheorem{fact}{Observation}[section]

\usepackage{bigfoot}
\DeclareNewFootnote[para]{default}
\DeclareNewFootnote{A}[arabic]
\DeclareNewFootnote[para]{B}[fnsymbol]
\makeatletter
\def\moverlay{\mathpalette\mov@rlay}
\def\mov@rlay#1#2{\leavevmode\vtop{%
    \baselineskip\z@skip \lineskiplimit-\maxdimen
    \ialign{\hfil$\m@th#1##$\hfil\cr#2\crcr}}}
\newcommand{\charfusion}[3][\mathord]{
  #1{\ifx#1\mathop\vphantom{#2}\fi
    \mathpalette\mov@rlay{#2\cr#3}
  }
  \ifx#1\mathop\expandafter\displaylimits\fi}
\DeclareRobustCommand\bigop[1]{%
  \mathop{\vphantom{\sum}\mathpalette\bigop@{#1}}\slimits@
}
\newcommand{\bigop@}[2]{%
  \vcenter{%
    \sbox\z@{$#1\sum$}%
    \hbox{\resizebox{\ifx#1\displaystyle.9\fi\dimexpr\ht\z@+\dp\z@}{!}{$\m@th#2$}}%
  }%
}
\makeatother

\newcommand{\cupdot}{\charfusion[\mathbin]{\cup}{\cdot}}
\DeclareMathOperator{\bigcupdot}{\charfusion[\mathop]{\bigcup}{\cdot}}

\newcommand{\AX}[1]{\textnormal{#1}}

\DeclareMathOperator{\lca}{lca}
\DeclareMathOperator{\child}{child}
\DeclareMathOperator{\parent}{parent}

\DeclareMathOperator{\build}{\mathtt{BUILD}}
\DeclareMathOperator{\bmg}{BMG}
\DeclareMathOperator{\qbmg}{qBMG}

\newcommand{\Rbin}{\mathop{\mathscr{R}^{\textrm{B}}}}
\newcommand{\hourglass}{\mathrel{\text{\ooalign{$\searrow$\cr$\nearrow$}}}}
\newcommand{\MOD}[1]{{[#1]}}

\providecommand{\keywords}[1]{\textbf{\textit{Keywords: }} #1}

\title{Quasi-Best Match Graphs}

\author[1]{Annachiara Korchmaros}
\author[1,4]{David Schaller}
\author[2]{Marc Hellmuth}
\author[1,3-7]{Peter F.\ Stadler}

\affil[1]{Bioinformatics Group, Department of Computer Science \&
  Interdisciplinary Center for Bioinformatics, Universit{\"a}t Leipzig,
  H{\"a}rtelstra{\ss}e 16-18, D-04107 Leipzig, Germany
  \authorcr \texttt{annachiara@bioinf.uni-leipzig.de} $\cdot$
  \texttt{sdavid@bioinf.uni-leipzig.de} $\cdot$
  \texttt{studla@bioinf.uni-leipzig.de}}

\affil[2]{Department of Mathematics, Faculty of Science,
  Stockholm University, SE-10691 Stockholm, Sweden
  \authorcr \texttt{marc.hellmuth@math.su.se}}

\affil[3]{German Centre for Integrative Biodiversity Research
  (iDiv) Halle-Jena-Leipzig, Competence Center for Scalable Data Services
  and Solutions Dresden-Leipzig, Leipzig Research Center for Civilization
  Diseases, and Centre for Biotechnology and Biomedicine at Leipzig
  University at Universit{\"a}t Leipzig}

\affil[4]{Max Planck Institute for Mathematics in the Sciences,
  Inselstra{\ss}e 22, D-04103 Leipzig, Germany}

\affil[5]{Institute for Theoretical Chemistry, University of Vienna,
  W{\"a}hringerstrasse 17, A-1090 Wien, Austria}

\affil[6]{Facultad de Ciencias, Universidad Nacional de Colombia, Bogot{\'a},
  Colombia}

\affil[7]{Santa Fe Institute, 1399 Hyde Park Rd., Santa Fe NM 87501,
  USA}

\date{\ }

\begin{document}
  
  \maketitle 
  
  \abstract{
    Quasi-best match graphs (qBMGs) are a hereditary class of directed,
    properly vertex-colored graphs. They arise naturally in mathematical
    phylogenetics as a generalization of best match graphs, which formalize
    the notion of evolutionary closest relatedness of genes (vertices) in
    multiple species (vertex colors). They are explained by rooted trees
    whose leaves correspond to vertices. In contrast to BMGs, qBMGs represent
    only best matches at a restricted phylogenetic distance. We provide
    characterizations of qBMGs that give rise to polynomial-time recognition
    algorithms and identify the BMGs as the qBMGs that are
    color-sink-free. Furthermore, two-colored qBMGs are characterized as
    directed graphs satisfying three simple local conditions, two of which
    have appeared previously, namely bi-transitivity in the sense of Das
    \emph{et al.} (2021) and a hierarchy-like structure of out-neighborhoods,
    i.e., $N(x)\cap N(y)\in\{N(x),N(y),\emptyset\}$ for any two vertices $x$
    and $y$. Further results characterize qBMGs that can be explained by
    binary phylogenetic trees.
  }

  \bigskip
  \noindent

\sloppy

\section{Introduction}

Best match graphs (BMGs) appear in mathematical phylogenetics to formalize
the notion of evolutionary closest relatives (homologs) of a gene in a
different species~\cite{Geiss:19a}. Phylogenetic relatedness derives
from a tree $T$ that describes the evolutionary history of the ``taxa''
(e.g.\ genes or species) at the leaves. In our setting, each leaf of $T$
represents an extant gene, and we assume to have additional knowledge of
the species $\sigma(x)$ in which each gene $x$ is found.  Given $T$, a
leaf $x$ is more closely related to $y$ than to $z$, if and only if the
lowest common ancestor $\lca(x,y)$ is a proper descendant of $\lca(x,z)$
in $T$.  Moreover, $xy$ is a best match if there is no gene $y'$ in the
same species as $y$, i.e., $\sigma(y')=\sigma(y)$, which is more closely
related to $x$. Best match graphs (BMGs) collect the best match
information for a family of related genes from different species.

BMGs have been studied in some
detail~\cite{Geiss:19a,Geiss:20b,Schaller:21b,Schaller:21c,Schaller:21d}
because of their close connection to the practically important problem of
orthology detection. Two genes $x$ and $y$ from different species
$\sigma(x)\ne\sigma(y)$ are orthologs if their last common ancestor
coincides with the divergence of the two species $\sigma(x)$ and
$\sigma(y)$ in which they reside~\cite{Fitch:70,Fitch:00}. In the absence
of horizontal gene transfer, orthologs are reciprocal best matches, i.e.,
$x$ is a best match of $y$, and $y$ is a best match of
$x$~\cite{Geiss:19a}. Identifying orthologs is an important task in
several areas of computational biology. In genome annotation, orthologs
are of interest because they are expected to have analogous functions in
different species. In contrast, homologous genes with a less direct
relationship are usually expected to have similar but distinct
functions~\cite{Koonin:05,Altenhoff:12,Gabaldon:13}. In phylogenomics, a
large collection of unrelated groups of orthologous genes are used
because their phylogenetic tree is nearly identical to the phylogenetic
tree representing the evolutionary history of the underlying
species~\cite{Delsuc:05,Young:19}. In structural biology, amino acids
conserved among orthologs but differing between paralogs (genes that
diverged from a duplication event) are used to identify sites that
determine the specificity of protein interactions~\cite{Mirny:02}. A
diverse class of widely used orthology detection tools is based on the
concept of reciprocal best matches,
see~\cite{Nichio:17,Rusin:14,Setubal:18a} for reviews.

\begin{figure}[t]
  \centering
  \includegraphics[width=0.6\linewidth]{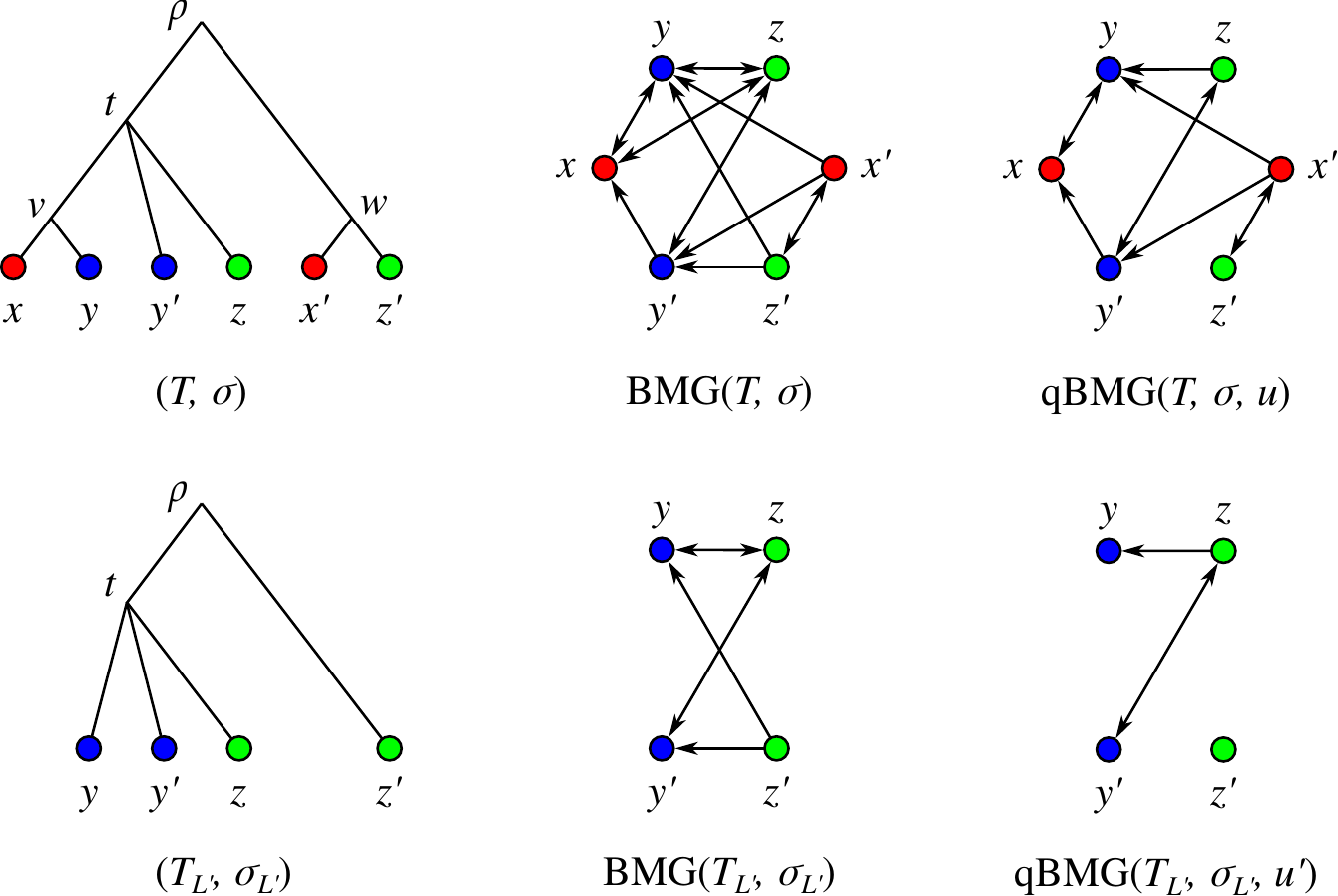}
  \caption{\emph{Upper Row:} The leaf-colored tree $(T,\sigma)$ gives
    rise to the best match graph $\bmg(T,\sigma)$ whose directed edges
    denote best matches. Reciprocal best matches correspond to
    bi-directional arcs (double-headed arrows). Ignoring best matches
    that are too distant yields a subgraph of $\bmg(T,\sigma)$. The notion of 
    too distant is formalized by a truncation map $u$ that
    depends on a query gene $q$ and a target species (color)
    $\alpha$. The quasi-best match graph $\qbmg(T,\sigma,u)$ contains
    only directed edges $xy$ if $y$ is a best match for $x$ and
    $\lca(x,y)$ is a descendant or equal to the detection limit
    $u(x,\alpha)$ with $\alpha=\sigma(y)$. Here, the detection limits are
    defined as follows: $u(x,\sigma(z))=u(y,\sigma(z))=v$,
    $u(z,\sigma(x))=z$, and $u(z',\sigma(y))=w$, whereas $u(q,\alpha)$ is
    the root $\rho$ if $\sigma(q)\ne\alpha$ for the other combinations of
    leaves and colors.\newline \emph{Lower Row:} Shown is the
    restricted leaf-colored tree $(T_{L'},\sigma_{L'})$ where $L'$
    consists of all vertices of color $\sigma(y)$ and $\sigma(z)$.  The
    tree $(T_{L'},\sigma_{L'})$ is obtained from $(T,\sigma)$ by removing
    the vertices $x$ and $x'$ as well as suppression of the vertices $v$
    and $w$ while keeping the colors of all remaining leaves. The
    BMG explained by $(T_{L'},\sigma_{L'})$ is an induced subgraph 
    $\bmg(T,\sigma)$. For qBMGs, the subgraph relationship also depends on 
    the choice of the truncation map $u'$. Choosing $u'$ as in 
    Equation~\ref{eq:restriction-u} ensures that 
    $\qbmg(T_{L'},\sigma_{L'},u')$ is an induced subgraph of
    $\qbmg(T,\sigma,u)$. In this example, $u'$ maps all combinations of
    leaves and colors to $\rho$ except for $u'(y,\sigma(z))=y$ and
    $u'(z',\sigma(y))=z'$.}
  \label{fig:qBMG-example}
\end{figure}

A complication in determining best matches arises when genes are so
distant that evolutionary relatedness, i.e., homology, cannot be
determined unambiguously. This limit on evolutionary distance will, in
general, depend both on the query gene $x$ and on the target genome or,
equivalently, a species $\alpha$, because significant sequence
similarities are easier to detect for large query sequences and in
smaller target genomes~\cite{Karlin:90}.  Thus best matches $y$ of $x$
are only considered if $\lca(x,y)$ is not closer to the root of $T$ than
the ``detection limit'' $u(x,\alpha)$ for target genome
$\sigma(y)=\alpha$. The resulting \emph{quasi-best match graph} (qBMG) is
always a subgraph of the corresponding BMG, see
Figure~\ref{fig:qBMG-example} for an illustrative example. BMGs have useful 
properties, such as the existence of a unique ``least
resolved tree'' that explains a given BMG~\cite{Geiss:19a}. Furthermore,
the least resolved tree of a BMG is displayed by the true phylogenetic
tree from which the BMG was derived. Here, we will be concerned with whether 
such properties generalize to the more general qBMGs,
which appear instead of the simpler BMGs, particularly when gene
families are considered at large phylogenetic scales.

An interest in qBMGs can also be motivated by studying classes of
(uncolored) digraphs defined in terms of constraints on neighborhoods. A
key result of~\cite{Geiss:19a} characterizes two-colored BMGs as
bipartite graphs whose out-neighborhoods satisfy four properties, of
which two have also been the subject of investigations in quite
different contexts in graph theory. Sink-free graphs have been studied in
the context of certain partition problems~\cite{Adam:84}, the construction
of orientations on graphs~\cite{Cohn:02}, and in conjunction with a variety
of algebraic structures~\cite{Abrams:10}. Bitransitive graphs were
introduced in~\cite{Das:21} in the context of
bitournaments. In~\cite{Korchmaros:21a,Korchmaros:20b} the condition that
$G$ is sink-free was lifted. As we shall see below, this also leads to
qBMGs in a very natural way.

This contribution is organized as follows: After introducing basic concepts
and fixing the notation in Section~\ref{sect:notation}, we provide a formal
definition of qBMGs in Section~\ref{sec:quasi-bmg} and derive their most
basic properties. In particular, we show that the qBMGs are exactly the
graphs that are obtained from BMGs by deleting, for an arbitrary given list
of pairs $(x,s)$ of vertices $x$ and colors $s$, all edges to out-neighbors
of $x$ with color $s$. In contrast to BMGs, the qBMGs form a hereditary
graph class. In Section~\ref{sect:triples}, a characterization of qBMGs
with an arbitrary number of colors given in terms of consistency of
informative and forbidden triples. This, in particular, allows us to
identify the BMGs as the qBMGs for which each gene $x$ has some
out-neighbor $y$ for every species $\sigma(y)$ that is distinct from
$\sigma(x)$, i.e., as the ``color-sink-free qBMGs''.
Section~\ref{sec:lrt} and~\ref{sect:be-qBMG} are then concerned with
questions related to least-resolved and binary trees to explain qBMGs.
Section~\ref{sect:2qBMG} is devoted to the 2-colored case. We
establish the equivalence between 2-qBMGs and the graphs studied in
~\cite{Korchmaros:21a,Korchmaros:20b}, obtain a simple characterization that
exposes the hierarchy-like structure of the out-neighborhoods, and derive
an explicit construction for an explaining tree and truncation map. We also
characterize 2-qBMGs in a setting where only the graph but not the
coloring is given.

\section{Notation and Preliminaries}
\label{sect:notation} 
\subsection{Graphs, Trees, and Vertex Colorings}

\paragraph{Vertex-Colored Graphs}
In this paper we consider simple, directed graphs $G=(V,E)$ with vertex set
$V=V(G)$ and edge set $E=E(G)$ unless $G$ is explicitly specified as
undirected. A graph is simple if it has neither parallel edges nor loops. A
graph $G'$ is a subgraph of $G$, in symbols $G'\subseteq G$ if
$V(G')\subseteq V(G)$ and $E(G')\subseteq E(G)$. For $W\subseteq V$ we
write $G[W]$ for the subgraph induced by $W$, i.e., $V(G[W])=W$ and
$xy\in E(G[W])$ if and only if $x,y\in W$ and $xy\in E(G)$.

We denote the set of \emph{out-neighbors} of $x$ in a graph $G$ by
$N_G(x)\coloneqq\{y\in V(G)\mid xy\in E(G)\}$. Analogously, we write
$N_G^-(x)\coloneqq\{y\in V(G)\mid yx\in E(G)\}$ for the set of in-neighbors
of a vertex $x\in V(G)$. If the context is clear, we omit the subscript
``$_G$''.  Since $G$ is simple, we have $x\notin N(x)$ and
$x\notin N^{-}(x)$. A \emph{source} is a vertex $x$ with in-degree zero,
that is, $N^{-}(x)=\emptyset$, while a vertex with out-degree zero, that
is, $N(x)=\emptyset$, is called \emph{sink}. A graph $G$ is
\emph{sink-free} if $N(x)\ne\emptyset$ for all $x\in V(G)$. Two vertices
$x,y$ are \emph{independent} in $G$ if
$\{x\}\cap N(y)=\{y\}\cap N(x)=\emptyset$, that is, neither $xy\in E(G)$,
nor $yx\in E(G)$. Following the usual convention, we extend the notation
for out-neighbourhood of a single vertex in $V$ to sets of vertices by
setting
\begin{equation}
  N:2^{V}\to 2^V,\quad A\mapsto N(A) \coloneqq\bigcup_{x\in A} N(x).
\end{equation}
For simplicity we write $N(x)$ instead of $N(\{x\})$. By construction,
$A\subseteq B$ implies $N(A)\subseteq N(B)$, i.e., $N$ is an \emph{isotonic
  map}~\cite{Hammer:64}.

A vertex coloring is a map $\sigma\colon V(G)\to S$, where $S$ is a finite
set with $|S|>1$. The coloring $\sigma$ is proper if $\sigma(x)=\sigma(y)$
implies $xy,yx\notin E(G)$. For every color $s\in S\setminus \{\sigma(x)\}$
we write $N(x,s)$ and ${N}^-(x,s)$ to denote the sets of out-neighbors and
in-neighbors of $x$ with color $s$, respectively. The tuple $(G,\sigma)$
denotes a graph $G$ together with a vertex coloring
$\sigma\colon V(G)\to S$. Moreover, we write $(G,\sigma)[W]$ for the
subgraph of a colored graph induced by $W\subseteq V$ and denote the
restriction of the coloring $\sigma$ to $W$ by $\sigma_W$. We say that
$(G,\sigma)$ is \emph{color-sink-free} if $N(x,s)\ne\emptyset$ for all
$x\in V(G)$ and all $s\in \sigma(V)\setminus\{\sigma(x)\}$. 

\paragraph{Leaf-Colored Rooted Trees}
A tree is an undirected, connected graph that does not contain cycles. We
also write $uv$ for the undirected edges in a tree $T$. Let $T=(V,E)$ be a
tree with root $\rho$ and leaf set $L\coloneqq L(T)\subset V$. The set of
inner vertices of $T$ is $V\setminus L$.  In case $V\neq L$, $\rho$ is
always assumed to be an inner vertex. An edge $e=uv\in E$ is an
\emph{inner} edge of $T$ if $u$ and $v$ are both inner vertices. Otherwise,
it is an \emph{outer} edge. A vertex $u\in V$ is an \emph{ancestor} of
$v\in V$ in $T$ if $u$ lies on the path from $\rho$ to $v$. In this case,
we write $v\preceq_T u$. Whenever we write $uv\in E$ for an edge in $T$ we
assume $v\prec_T u$. A vertex $v$ is a \emph{child} of $u$ if $uv\in E$. We
write $\child_T(u)$ for the set of children of $u$ in $T$. If $v$ is a
child of $u$, then $u$ is the (unique) parent of $v$; we write
$u=\parent_T(v)$. Note that $\parent_T(v)$ is defined for all vertices 
$v\in V\setminus\{\rho\}$.
The subtree of $T$ rooted at $u$ and its leaf set are
denoted by $T(u)$ and $L(T(u))$, respectively. A tree $T$ is
\emph{phylogenetic} if $u\in V(T)$ is either a leaf or $u$ has
$|\child(u)|\ge 2$ children. \emph{All trees appearing in this contribution
  are assumed to be phylogenetic.} The \emph{least common ancestor}
$\lca_{T}(A)$ is the unique $\preceq_T$-smallest vertex that is an ancestor
of all vertices in $A\subseteq V$. For brevity, we write $\lca_{T}(x,y)$
instead of $\lca_{T}(\{x,y\})$. The subtree of $T$ rooted in a vertex $u\in
V$ is denoted by $T(u)$.

Suppressing a vertex $v$ of degree two refers to the operation of
removing $v$ and its two incident edges $uv$ and $vw$ and adding the edge
$uw$.  Following~\cite{Bryant:95}, $T_{L'}$ denotes the
\emph{restriction} of $T$ to a subset $L'\subseteq L(T)$, i.e.,
$T_{L'}$ is obtained from the (unique) minimal subtree of $T$
connecting all leaves in $L'$ by subsequently suppressing all vertices with
degree two except possibly the root $\rho_{T_{L'}}=\lca_T(L')$. Note that
$\rho_{T_{L'}}$ is not necessarily the original root $\rho_T$. The
contraction of an inner edge $e=uv\in E(T)$ is an operation that produces a
new tree $T_e$ with vertex set
$V(T_e)=(V(T)\setminus \{u,v\}) \cup \{w\}$, where the new
vertex $w$ replaces both $u$ and $v$. Thus we have
$\child_{T_e}(w)=\child_T(u)\cup\child_T(v)$ and
$\parent_{T_e}(w)=\parent_T(u)$. We say that $T$ \emph{displays} or
\emph{is a refinement of} a tree $T'$, in symbols $T'\le T$, if $T'$
is obtained from a restriction $T_{L'}$ of $T$ after a (possibly empty) 
sequence of inner edge contraction. We write $T'< T$ for $T'\le T$ and 
$T'\ne T$.

A \emph{leaf coloring} of a tree is a map $\sigma\colon L\to S$ where $S$
is a non-empty set of \emph{colors}.  We consider leaf-colored trees
$(T,\sigma)$ and write $\sigma(L')\coloneqq \{\sigma(v)\mid v\in L'\}$ for
subsets $L'\subseteq L$. We say that $(T',\sigma')$ is displayed by
$(T,\sigma)$ if $T'\le T$ and $\sigma(v)=\sigma'(v)$ for all $v\in L(T')$.

\paragraph{Rooted Triples}
A \emph{(rooted) triple} is a tree on three leaves with two inner vertices.
We write $ab|c$ (or equivalently $ba|c$) for the triple $t$ in which the
root has two children $c$ and $\lca_t(a,b)$.  Accordingly, we say that a
tree $T$ \emph{displays} a triple $ab|c$ if $a,b,c\in L(T)$ are pairwise
distinct and $\lca_T(a,b)\prec_T\lca_T(a,c)=\lca_T(b,c)$.  An edge $uv$ in
a tree $T$ is \emph{distinguished} by a triple $ab|c$ if and only if
$\lca_{T}(a,b)=v$ and $\lca_{T}(\{a,b,c\})=u$.  A set $\mathscr{R}$ of
triples is \emph{consistent} if there is a tree displaying all triples in
$\mathscr{R}$.  Given a consistent set $\mathscr{R}$ of triples defined on
a leaf set $L$, the algorithm \texttt{BUILD} returns a tree on $L$, denoted
by $\build(\mathscr{R}, L)$, that displays all triples in
$\mathscr{R}$~\cite{Aho:81}.  We furthermore extend the definition of
consistency to pairs of sets of required and forbidden triples, i.e., we
call a pair $(\mathscr{R}, \mathscr{F})$ of two triple sets
\emph{consistent} if there is a tree $T$ that displays all
triples in $\mathscr{R}$ but none of the triples in $\mathscr{F}$. In this
case we say that $T$ \emph{agrees} with $(\mathscr{R}, \mathscr{F})$. The
restriction of a triple set $\mathscr{R}$ to a set of leaves $L'$ is given
by $\mathscr{R}_{L'}\coloneqq\{ab|c\in\mathscr{R}\mid a,b,c\in L'\}$.

\paragraph{Hierarchies}
Let $X$ be a non-empty finite set. A system $\mathcal{H}\subseteq 2^X$ of
non-empty sets is a \emph{hierarchy} on $X$ if (i) $U,V\in\mathcal{H}$
implies $U\cap V\subseteq \{\emptyset, U, V\}$, (ii) $\{x\}\in\mathcal{H}$
for all $x\in X$, and (iii) $X\in\mathcal{H}$. There is a 1-1
correspondence between hierarchies on $L$ and phylogenetic rooted trees
with leaf set $L$ given by
$\mathcal{H}(T)\coloneqq\{ L(T(v))\mid v\in V(T)\}$, see
~\cite[Theorem~3.5.2]{Semple:03}.  The inverse map is obtained by the Hasse
diagram $T(\mathcal{H})$ of $\mathcal{H}$ with respect to set
inclusion.  A set system $\mathcal{H}\subseteq2^X$ is \emph{hierarchy-like}
if it satisfies condition (i). In this case, $\mathcal{H}$ still
corresponds to a forest. The roots of the constituent trees correspond to
the inclusions-maximal elements of $\mathcal{H}$, and the leaves correspond
to the inclusions-minimal sets, which correspond to the singletons if and
only if (ii) holds.  For simplicity of notation, we will think of the
leaves of $T$ as elements $v\in X$ rather than singleton sets $\{v\}\in2^X$.

\subsection{Best Match Graphs}

\begin{definition}\label{def:bm}
  Let $(T,\sigma)$ be a leaf-colored tree. A leaf $y\in L(T)$ is a
  \emph{best match} of the leaf $x\in L(T)$ if $\sigma(x)\neq\sigma(y)$ and
  $\lca(x,y)\preceq_T \lca(x,y')$ holds for all leaves $y'$ of color
  $\sigma(y')=\sigma(y)$.
  \label{def:BMG}
\end{definition}
Given $(T,\sigma)$, the \emph{best match graph} (BMG) associated to
$(T,\sigma)$ is the digraph $\bmg(T,\sigma) = (V,E)$ with vertex set
$V=L(T)$, vertex coloring $\sigma$, and $xy\in E$ if and only if $y$ is a
best match of $x$ with respect to $(T,\sigma)$~\cite{Geiss:19a}.
\begin{definition}\label{def:BestMatchGraph}
  An arbitrary vertex-colored digraph $(G,\sigma)$ is a \emph{best match
    graph (BMG)} if there exists a leaf-colored tree $(T,\sigma)$ such that
  $(G,\sigma) = \bmg(T,\sigma)$. In this case, we say that $(T,\sigma)$
  \emph{explains} $(G,\sigma)$.
\end{definition}
A BMG $(G,\sigma)$ with vertex set $L$ is an $\ell$-BMG if
$|\sigma(L)|=\ell$. Figure~\ref{fig:qBMG-example} in the introduction
shows an example of a 3-BMG on $6$ vertices, \cite[Figure~7]{Korchmaros:21a} 
gives an example of a 2-BMG on $10$ vertices 
and~\cite[Figure~2(a)]{Korchmaros:21a} gives an example of a 2-BMG on $11$ 
vertices. Since BMGs are properly
colored, every 2-BMG is bipartite. Furthermore, the subgraph of a BMG
induced by the subset $L_{rs}\coloneqq\left\{x\in L\mid
\sigma(x)\in\{r,s\}\right\}$ of $L$, i.e., the vertices with two
colors, is always a 2-BMG~\cite{Geiss:19a}.

Now, we briefly review characterizations of BMGs that will be of
relevance for this contribution. 
\begin{proposition}[{\cite[Theorem~4]{Geiss:19a}} and~\cite{Schaller:21b}]
  \label{prop:char2-BMG}
  A properly two-colored graph $(G,\sigma)$ is a 2-BMG if and only if it is
  (i) sink-free and (ii) the out-neighborhoods satisfy the following three
  conditions:
  \begin{itemize} \label{Ns}
    \item[\AX{(N1)}] $x\notin N(y)$ and $y\notin N(x)$ implies
    $N(x)\cap N(N(y))=N(y)\cap N(N(x))=\emptyset$.
    \item[\AX{(N2)}] $N(N(N(x))) \subseteq N(x)$.
    \item[\AX{(N3')}] If $x\notin N(N(y))$, $y\notin N(N(x))$, and
    $N(x)\cap N(y)\neq \emptyset$ then we have
    $N^-(x)= N^-(y)$ and one of the inclusions
    $N(x)\subseteq N(y)$ or $N(y)\subseteq N(x)$.
  \end{itemize}
\end{proposition}

Properties \AX{(N1)} and \AX{(N2)} may be rephrased as
in~\cite{Korchmaros:21a,Korchmaros:20b}:
\begin{itemize} 
  \item[\AX{(N1)}] If $x$ and $y$ are independent vertices (or $x=y$),
  then there exist no vertices $w$ and $t$ such that $xt, yw, tw \in E(G)$.
  \item[\AX{(N2)}] If there are vertices $x_1,x_2,y_1,y_2\in L$ with
  $x_1 y_1, y_1 x_2, x_2 y_2 \in E(G)$, then $x_1 y_2 \in E(G)$.
\end{itemize}
These two properties imply that some directed edges must be present in
a 2-BMG; see~\cite[Figure~1]{Korchmaros:21a} and
\cite[Figure~6]{Korchmaros:21a} for the edges required by \AX{(N2)} and
the counterpart of \AX{(N1)}, respectively. Property \AX{(N2)} was
introduced as \emph{bi-transitivity} in ~\cite{Das:21} and investigated in
relation to topological orderings in
\cite{Korchmaros:21a,Korchmaros:20b}. It suffices to require \AX{(N1)} for
$x\ne y$ because, for $x=y$, $N(x)\cap N(N(x))=\emptyset$ is an
immediate consequence of the fact that $G$ is assumed to be bipartite.
Nevertheless, we define \AX{(N1)} here for all $x,y\in V(G)$. We shall
see below in Theorem~\ref{thm:bipartite} that this choice of the axiom
allows us to remove the explicit requirement that $G$ is bipartite.

Properties \AX{(N1)}, \AX{(N2)} and \AX{(N3')} were translated into a
system of forbidden induced subgraphs of a 2-BMG;
see~\cite{Schaller:21b}. In particular, the first part of \AX{(N3')}, i.e.,
the implication $x\notin N(N(y))$, $y\notin N(N(x))$ and $N(x)\cap N(y)\neq
\emptyset \implies N^-(x)= N^-(y)$, is already covered by \AX{(N1)}. This
makes it possible to replace \AX{(N3')} by the simpler condition
\begin{itemize}
  \item[(N3)] If $N(x)\cap N(y)\ne\emptyset$ then $N(x)\subseteq N(y)$ or
  $N(y)\subseteq N(x)$.
\end{itemize} 
For completeness, we prove directly that \AX{(N3')} may be replaced by 
\AX{(N3)}.
\begin{lemma}
  \label{lem:N1N2N3-old-eq-new}
  A properly two-colored graph $(G,\sigma)$ satisfies \AX{(N1)}, \AX{(N2)},
  and \AX{(N3')} if and only if it satisfies \AX{(N1)}, \AX{(N2)}, and
  \AX{(N3)}.
\end{lemma}
\begin{proof}
  Suppose $(G,\sigma)$ satisfies \AX{(N1)}, \AX{(N2)}, and \AX{(N3')} and
  assume, for contradiction, that \AX{(N3)} is violated, i.e., that there
  are vertices $x,y\in V(G)$ with $N(x)\cap N(y)\ne\emptyset$ such that
  neither $N(x)\subseteq N(y)$ nor $N(y)\subseteq N(x)$ is true.  Hence,
  $x\neq y$.  Thus, there are distinct vertices $u\in N(x)\setminus N(y)$,
  $v\in N(y)\setminus N(x)$. By assumption, there exists $z\in N(x)\cap
  N(y)$.  Since $(G,\sigma)$ is properly colored,
  $\sigma(x)=\sigma(y)\ne\sigma(z)=\sigma(u)=\sigma(v)$, and $x$, $y$, $z$,
  $u$, and $v$ are pairwise distinct.  If $x\notin N(N(y))$ and $y\notin
  N(N(x))$, then we immediately obtain a contradiction to \AX{(N3')}.
  Hence, assume that there exists $w\in V(G)$ such that $xw,wy\in E(G)$,
  i.e., $y\in N(N(x))$.  Since $w\in N(x)$, $v\notin N(x)$ and $(G,\sigma)$
  is properly colored, we have $w\notin \{x,y,v\}$.  However, $xw,wy,yv \in
  E(G)$ together with \AX{(N2)} imply that $xv\in E(G)$ and thus, $v\in
  N(x)$; a contradiction.  An analogous contradiction is obtained for that
  case $x\in N(N(y))$.
  
  Now suppose $(G,\sigma)$ satisfies \AX{(N1)}, \AX{(N2)}, and \AX{(N3)}
  and assume, for contradiction, that $(G,\sigma)$ does not satisfy
  \AX{(N3')}.  Hence, there are two vertices $x,y\in V(G)$ with
  $x\notin N(N(y))$, $y\notin N(N(x))$, and $N(x)\cap N(y)\ne\emptyset$
  such that (i) neither of the inclusions $N(x)\subseteq N(y)$ and
  $N(y)\subseteq N(x)$ holds, or (ii) $ N^-(x)\ne N^-(y)$.  Case~(i)
  immediately contradicts \AX{(N3)}. In case~(ii), i.e.,
  $ N^-(x)\ne N^-(y)$, we have $x\ne y$ and we can assume that
  there exists $w\in V(G)$ with $wx\in E(G)$ and $wy\notin E(G)$.  We also
  have $yw\notin E(G)$ since otherwise $yw,wx\in E(G)$ would contradict
  $x\notin N(N(y))$. If $w=y$ then $x$ and $y$ are adjacent, and thus
  $N(x)\cap N(y)=\emptyset$, contradicting our assumption.  Hence, $y$ and
  $w$ are independent in $G$.  From $N(x)\cap N(y)\ne\emptyset$, it follows
  that $x$ and $y$ have a common out-neighbor $z$.  Since $(G,\sigma)$ is
  properly colored and $yz\in E(G)$ but $yw\notin E(G)$, we can infer that
  that the four vertices $x$, $y$, $z$, and $w$ are pairwise distinct.  In
  summary, $(G,\sigma)$ contains the two independent vertices $w$ and $y$
  and the two vertices $x$ and $z$ with $wx,yz,xz\in E(G)$; this
  contradicts \AX{(N1)}.  Hence, we conclude that $(G,\sigma)$ also
  satisfies \AX{(N3')}.
\end{proof}
We remark that \AX{(N1)} together with \AX{(N3)} implies \AX{(N3')}, but
the converse is not true in general unless \AX{(N2)} is assumed.  We note
that Condition \AX{(N3)} may be written also as $N(x)\cap
N(y)\in\{\emptyset, N(x), N(y)\}$, i.e.\ the out-neighborhoods form a
hierarchy-like set system on $V(G)$. It is not a hierarchy because
$N(x)\ne V(G)$ for all $x\in V(G)$ since $(G,\sigma)$ is bipartite.

In~\cite{Schaller:21b,Schaller:21c,Schaller:21d}, BMGs have been
characterized in terms of certain sets of rooted triples that easily can be
constructed for every given vertex-colored digraph.
\begin{definition}\label{def:informative_triples}
  Let $(G,\sigma)$ be a vertex-colored digraph. Then the set of
  \emph{informative triples} is
  \begin{align*}
    \mathscr{R}(G,\sigma) \coloneqq \big\{&
    ab|b' \colon
    \sigma(a)\neq\sigma(b)=\sigma(b'),\,
    ab\in E(G), \text{ and }
    ab'\notin E(G) \big\},
    \intertext{and the set of \emph{forbidden triples} is}
    \mathscr{F}(G,\sigma) \coloneqq \big\{
    &ab|b' \colon
    \sigma(a)\neq\sigma(b)=\sigma(b'),\,
    b\ne b',\, \text{ and }
    ab,ab'\in E(G) \big\}.
    \intertext{For vertex-colored digraphs that are associated with binary 
      trees, we will furthermore need}
    \Rbin(G,\sigma) \coloneqq
    &\mathscr{R}(G,\sigma) \cup 
    \{bb'|a\colon ab|b'\in \mathscr{F}(G,\sigma)\}
  \end{align*}
\end{definition}
Note that $a,b,b'$ must be pairwise distinct whenever
$ab|b'\in\mathscr{R}(G,\sigma)$, $ab|b'\in\mathscr{F}(G,\sigma)$, or
$bb'|a\in\Rbin(G,\sigma)$.  Moreover, the forbidden triples always come in
pairs, i.e., $ab|b'\in \mathscr{F}(G,\sigma)$ implies $ab'|b\in
\mathscr{F}(G,\sigma)$. The triples in each of the three sets are
2-colored because they refer to ($G,\sigma)[L_{\sigma(a)\sigma(b)}]$, i.e., an 
induced 2-BMG.

\begin{proposition}[{\cite[Lemma~3.4 and Theorem~3.5]{Schaller:21b}}]
  \label{prop:BMG-triple-charac}
  A properly colored graph $(G,\sigma)$ is a BMG if and only if (i) it is
  color-sink-free and (ii) $(\mathscr{R}(G,\sigma),\mathscr{F}(G,\sigma))$
  is consistent.  In this case, every tree $(T,\sigma)$ with leaf set
  $V(G)$ that agrees with $(\mathscr{R}(G,\sigma),\mathscr{F}(G,\sigma))$
  explains $(G,\sigma)$.
\end{proposition}

\emph{Binary-explainable} BMGs are the vertex-colored digraphs explained by 
binary trees.
\begin{proposition}[{\cite[Theorem~8]{Schaller:21c}}]
  A properly colored graph $(G,\sigma)$ is a binary-explainable BMG if and only 
  if (i) it is color-sink-free and (ii) $\Rbin\coloneqq \Rbin(G,\sigma)$ is 
  consistent.
  In this case, the BMG $(G,\sigma)$ is explained by every refinement of the 
  leaf-colored tree $(\build(\Rbin, V(G)), \sigma)$.
\end{proposition}

\section{Basic Properties of Quasi-Best Match Graphs}\label{sec:quasi-bmg}

We generalize the notion of best match graphs to quasi-best match graphs in 
order to address the limitations of methods to detect significant sequence 
similarities. To this end we introduce the notion of a ``relative detection 
limit'' beyond which best matches remain undetected and thus are assumed to be 
absent in the graphs of interest.

\begin{definition}\label{def:qbm} 
  Let $(T,\sigma)$ be a leaf-colored tree with vertex set $V$, leaf set $L$
  and $\sigma(L) \subseteq S$. A \emph{truncation map} $u_T\colon L\times
  S\to V$ assigns to every leaf $x\in L$ and color $s\in S$ a vertex of $T$
  such that $u_T(x,s)$ lies along the unique path from $\rho_T$ to $x$ and
  $u_T(x,\sigma(x))=x$.  A leaf $y \in L$ with color $\sigma(y)$ is a
  \emph{quasi-best match for $x\in L$} (with respect to $(T,\sigma)$
  and $u_T$) if both conditions (i) and (ii) are satisfied:
  \begin{itemize}
    \item[(i)] $y$ is a \emph{best match} of $x$ in $(T,\sigma)$.
    \item[(ii)] $\lca_T(x,y) \preceq u_T(x,\sigma(y))$.
  \end{itemize}
  The vertex-colored digraph $\qbmg(T,\sigma,u_T)$ on the vertex set $L$
  whose edges are defined by the quasi-best matches is the \emph{quasi-best
    match graph} (qBMG) of $(T,\sigma,u_T)$.
\end{definition}
Note that, since $xy$ is an edge only if $\sigma(x)\ne\sigma(y)$,
$\qbmg(T,\sigma,u_T)$ is always properly colored. Furthermore, if
$u_T(x,s)=x$, then $x$ has no out-neighbors of color $s\in S$ in
$\qbmg(T,\sigma,u_T)$, but  the converse is not true in general.

\begin{definition} \label{quasibmgDef} A vertex-colored digraph
  $(G,\sigma)$ with vertex set $V(G)=L$ is a \emph{(colored) quasi-best
    match graph} (qBMG) if there is a leaf-colored tree $(T,\sigma)$ and a
  truncation map $u_T$ on $(T,\sigma)$ such that
  $(G,\sigma)=\qbmg(T,\sigma,u_T)$.
\end{definition}

In analogy to BMGs, a qBMG $(G,\sigma)$ with vertex set $L$ is an
$\ell$-qBMG if $|\sigma(L)|=\ell$.  The \emph{trivial} truncation map is
defined as $u^{\rho}_T(x,s)\coloneqq \rho_T$ for all $x\in L(T)$ and all
$s\in S\setminus\{\sigma(x)\}$. In this case, condition (ii) becomes void so 
that the trivial truncation map has no
effect. We state this fact as a corollary.

\begin{corollary}\label{cor:trivialTmap}
  Let $(T,\sigma)$ be a leaf-colored tree and $u^{\rho}_T$ be the trivial
  truncation map.  Then $\bmg(T,\sigma) = \qbmg(T,\sigma,u^{\rho}_T)$, that
  is, $xy$ is a quasi-best match with respect to
  $(T,\sigma,u^{\rho}_T)$ if and only if it is a best match with respect to 
  $(T,\sigma)$. In particular, every BMG is a qBMG.
\end{corollary}

\begin{fact}\label{remark:subgraph}
  If $u_1(x,r)\preceq_T u_2(x,r)$ for all $x\in L$ and $r\in S$, then 
  $\qbmg(T,\sigma,u_1)$ is a subgraph of $\qbmg(T,\sigma,u_2)$.
\end{fact}
This, together with Corollary~\ref{cor:trivialTmap}, shows that every qBMG is a 
subgraph of a BMG.
\begin{lemma}\label{2out_neighbour}
  Let $(\tilde G, \sigma)\coloneqq \bmg(T,\sigma)$ 
  be a BMG. Let $(G,\sigma)\coloneqq \qbmg(T,\sigma,u_T)$ be the qBMG arising 
  from  $(T,\sigma)$ with a truncation map $u_T$. Then, for all $x\in L$ and 
  all colors $s\in \sigma(L)$, we have 
  either $N_{G}(x,s)= N_{\tilde G}(x,s)$ or $N_{G}(x,s)=\emptyset$.
\end{lemma}
\begin{proof}
  If $N_{G}(x,s)=\emptyset$, we are done. Hence, assume that
  $N_{G}(x,s)\ne\emptyset$.  From Definition~\ref{def:qbm}(i), every
  $y\in N_{G}(x,s)$ is a best match of $x$, thus
  $N_{G}(x,s) \subseteq N_{\tilde G}(x,s)$.  Moreover, there is a vertex
  $u\in V(T)$ such that $u=\lca_T(x,y)$ for all $y\in N_{G}(x,s)$ and
  $u\prec_T \lca(x,y')$ for all $y'\notin N_{G}(x,s)$ with
  $\sigma(y)=s$. In addition, we have $u\preceq_T u_T(x,s)$ by
  Definition~\ref{def:qbm}(ii).  The latter properties are, in particular, valid
  for every $\tilde y\in N_{\tilde G}(x,s)$, which implies that
  $N_{\tilde G}(x,s)\subseteq N_{G}(x,s)$ and therefore,
  $N_{G}(x,s)= N_{\tilde G}(x,s)$.
\end{proof}

A leaf-colored tree $(T,\sigma)$ explaining a BMG $(G,\sigma)$ is
\emph{least resolved} if there is no tree $T'<T$ such that
$(G,\sigma)=\bmg(T',\sigma)$.  By~\cite[Theorem~8]{Geiss:19a}, every BMG
has a unique least resolved tree (LRT) $\widehat{T}$.  This result can also
be used in the study of qBMGs, since, by the following lemma, every qBMG may
be assumed to arise from the LRT of some BMG by means of a truncation map.
\begin{lemma}
  \label{lem:hatT}
  Let $(T,\sigma)$ be a leaf-labeled tree with truncation map $u$ and let
  $(\widehat{T},\sigma)$ be the LRT of $\bmg(T,\sigma)$. Then there is a
  truncation map $\hat{u}$ on $\widehat{T}$ such that
  $\qbmg(T,\sigma,u)=\qbmg(\widehat{T},\sigma,\hat{u})$.
\end{lemma}
\begin{proof}
  Put $(\tilde G, \sigma)\coloneqq \bmg(T,\sigma)$ and
  $(G,\sigma)=\qbmg(T,\sigma,u)$.  Then, by definition,
  $(\tilde G, \sigma) = \bmg(\widehat T,\sigma)$. Consider the truncation
  function $\hat u$ for $\widehat{T}$ defined by
  $\hat u(x,s) = \rho_{\widehat T}$ if $N_{G}(x,s)\ne\emptyset$ and
  $\hat u(x,s) = x$ if $N_{G}(x,s) =\emptyset$.  Now, let
  $(\widehat{G}, \sigma)\coloneqq \qbmg(\widehat{T},\sigma,\hat{u})$.  By
  construction, $N_{\hat G}(x,s)= N_{\tilde G}(x,s)=N_{G}(x,s)$ if
  $N_{G}(x,s)\ne\emptyset$ and $ N_{\hat G}(x,s)=N_{G}(x,s)=\emptyset$
  otherwise. Since a digraph is uniquely determined by its out-neighbors,
  we have $\widehat{G}=G$.
\end{proof}
It should be noticed that Lemma~\ref{lem:hatT} does not imply $\widehat{T}$
to be a least resolved explanation for the qBMG.  We will study some
properties of least resolved trees of qBMGs in Section~\ref{sec:lrt}.

\begin{theorem}
  \label{thm:removeNs}
  A vertex-colored graph $(G,\sigma)$ with vertex set $L$ is a qBMG if and
  only if there is a BMG $(\tilde G,\sigma)$ such that, for all $x\in L$
  and $s\in S=\sigma(L)$, either $N_{G}(x,s)=N_{\tilde{G}}(x,s)$ or
  $N_{G}(x,s)=\emptyset$ holds.
\end{theorem}
\begin{proof}
  Lemma~\ref{2out_neighbour} implies the \emph{only-if}-direction. For the
  \emph{if}-direction, suppose that $(\tilde G,\sigma)$ is a BMG that is
  explained by $(T,\sigma)$.  Let $(G,\sigma)$ be a vertex-colored graph
  such that either $N_{G}(x,s)=N_{\tilde{G}}(x,s)$ or
  $N_{G}(x,s)=\emptyset$.  Then $u$ with $u(x,s)=\rho_{T}$ if
  $N_{G}(x,s)=N_{\tilde{G}}(x,s)$ and $u(x,s)=x$ if $N_{G}(x,s)=\emptyset$
  is a truncation map for $T$. By construction,
  $(G,\sigma) = \qbmg(T,\sigma,u)$ and thus, it is a qBMG.
\end{proof}

Theorem~\ref{thm:removeNs} gives a motivation for the following definition.
\begin{definition}\label{def:assocBMG}
  Let $(G,\sigma)$ be a qBMG with vertex set $L$. A BMG $(\tilde G,\sigma)$
  with vertex set $L$ is \emph{associated} with $(G,\sigma)$ if for all
  $x\in L$ and $s\in\sigma(L)$ either $N_{G}(x,s)=N_{\tilde{G}}(x,s)$ or
  $N_{G}(x,s)=\emptyset$ holds.
\end{definition}
A necessary condition for BMGs $(\tilde G,\sigma)$ to be associated with a
qBMG $(G,\sigma)$ is that $(G,\sigma)$ is a subgraph of
$(\tilde G,\sigma)$. In this case, $(G,\sigma)$ can be obtained from
$(\tilde G,\sigma)$ by removing all out-neighbors of specific colors for
specific vertices. This condition is not sufficient, however, as the
example in Figure~\ref{fig:akexample} shows.

\begin{figure}
  \centering
  \includegraphics[width=0.85\linewidth]{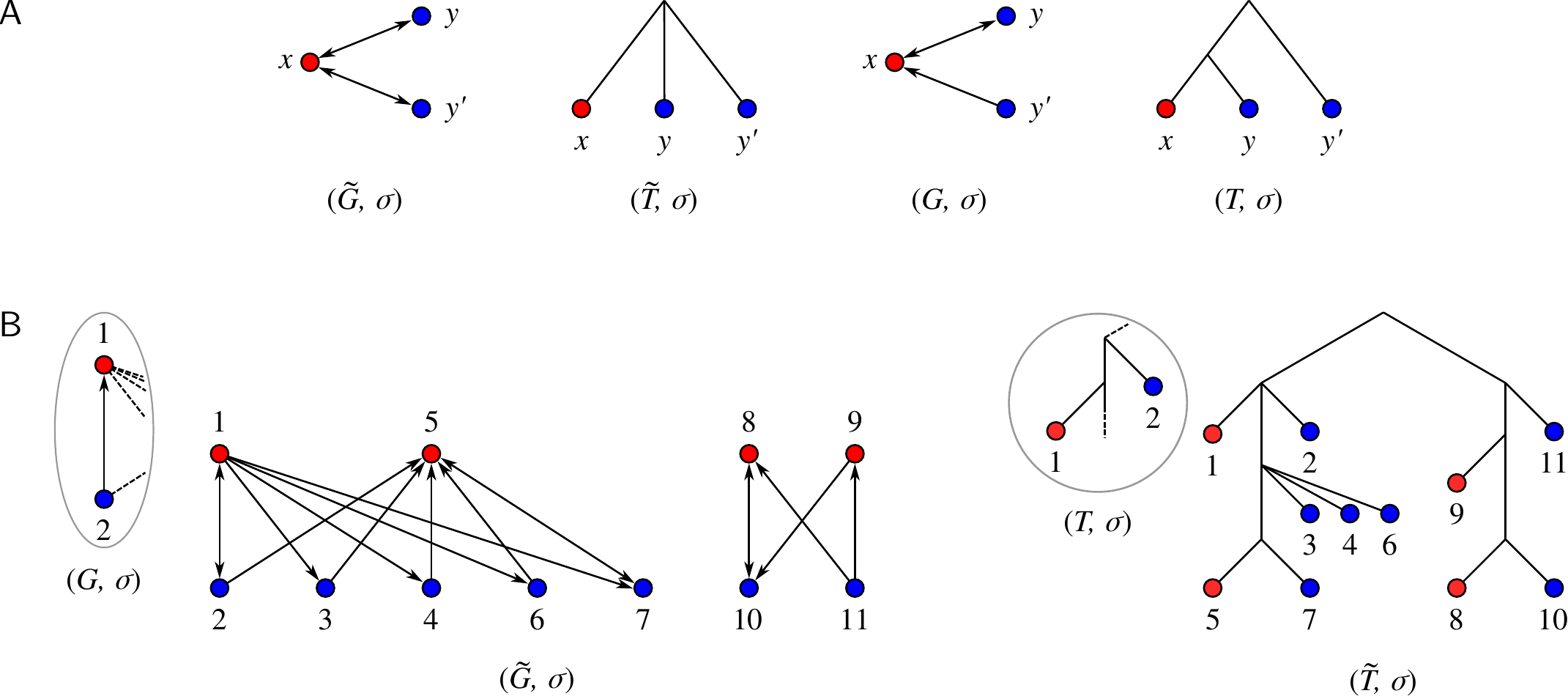}
  \caption{Examples of qBMGs $(G,\sigma)$ that are a subgraph of a BMG
    $(\tilde{G},\sigma)$ that is not associated with $(G,\sigma)$.  A: Both
    digraphs are 2-BMGs on three vertices since they are explained by the trees
    $(T,\sigma)$ and $(\tilde T,\sigma)$, respectively. $G$ arises from
    $\tilde{G}$ by deleting just one edge, namely $xy'$. Thus $G$ is a
    subgraph of $\tilde{G}$. However, $\tilde{G}$ is not associated with
    $G$ since $x$ is also the tail of edges other than $xy'$.  B: Both
    digraphs are again 2-BMGs on $11$ vertices explained by $(T,\sigma)$
    and $(\tilde T,\sigma)$, respectively.  The trees are identical up to
    one edge contraction, as indicated by the circle.}
  \label{fig:akexample}
\end{figure}

By Theorem~\ref{thm:removeNs}, every qBMG has at least one associated
BMG. It is, therefore, possible to obtain all qBMGs from the BMGs by removing
all out-edges with a subset of colors for a subset of the vertices.

In order to show that qBMGs form a hereditary graph class, we consider the
deletion of a single vertex and show that both $(G,\sigma)$ and the
explaining tree $(T,\sigma,u)$ are quite ``well-behaved'' when a single
vertex is deleted.
\begin{lemma}
  \label{lem:delvert}
  Let $(G,\sigma)=\qbmg(T,\sigma,u)$ be a qBMG and $v\in L$. Then there
  exists a truncation map $u'$ such that
  $(G,\sigma)[L']=\qbmg(T_{L'},\sigma_{L'},u')$ where
  $L'=L\setminus\{v\}$. In particular, the digraph $(G,\sigma)[L']$ is
  still a qBMG.
\end{lemma}
\begin{proof}
  Consider the phylogenetic tree $T'\coloneqq T_{L'}$ obtained by deleting
  the leaf $v$ and suppressing $w=\parent_T(v)$ if $v$ has a single sibling
  $v^*$, i.e., $\child_T(w)=\{v,v^*\}$, and $\sigma'(x)=\sigma(x)$ for
  $x\in L'$.  Note that if $w$ is the root of $T$, then $v^*$ is the new
  root of $T'$.  Depending on whether $w$ is suppressed or not, there is a
  1-1 correspondence between $V(T)\setminus\{v\}$ and $V(T')$ or
  $V(T)\setminus\{v,w\}$ and $V(T')$, respectively, such that
  $L(T')=L(T)\setminus\{v\}$. In our notation, we identify corresponding
  vertices of $T$ and $T'$.  Moreover, we will use the fact that the
  ancestor order $\preceq_{T'}$ is equal to $\preceq_{T}$ restricted to
  $V(T')$.  Given a truncation map $u$ on $T$, we define its restriction
  $u'$ on $T'$ for all $x\in L'$ and $s\in S$ by
  \begin{equation}
    \label{eq:restriction-u}
    u'(x,s) = \begin{cases}
      v^* & \text{if } \child_T(w)=\{v,v^*\} \text{ and thus, } u(x,s)\not\in 
      V(T') \\
      u(x,s) & \text{else if } s\ne \sigma(v) \text{ or } 
      N(x,s)\setminus\{v\}\neq \emptyset \\
      x & \text{otherwise} \\
    \end{cases}
  \end{equation}
  From now on we write $\sigma'\coloneqq \sigma_{L'}$.  The restricted
  leaf-colored tree $(T',\sigma')$ together with the truncation map $u'$
  explains $(G',\sigma')=\qbmg(T',\sigma',u')$. We next show that
  $(G',\sigma')$ is precisely the induced subgraph $(G,\sigma)[L']$.
  Denote the out-neighborhoods of $G$ and $G'$ by
  $N(x,s)\coloneqq N_{G}(x,s)$ and $N'(x,s)\coloneqq N_{G'}(x,s)$,
  respectively.  By construction, $u(x,\sigma(x))=u'(x,\sigma(x))=x$ and
  thus $N(x,\sigma(x))=N'(x,\sigma(x))=\emptyset$ for all $x\in L'$.  We
  proceed by comparing $N(x,s)$ and $N'(x,s)$ for $x\in L'$ and
  $s\in S\setminus \{\sigma(x)\}$. To this end, note that for each
  $y\in L'$, we have $\lca_T(x,y)=\lca_{T'}(x,y)$ since otherwise, we must
  have $\lca_T(x,y)=w\notin V(T')$, which would imply $y=v$; a
  contradiction.
  
  Suppose first that $v\notin N(x,s)$.  If $y\in N(x,s)$, then
  $\lca_T(x,y)\preceq_T u(x,s)$ by Definition~\ref{def:qbm}(ii), and
  there is no $y'\in L$ such that $\sigma(y)=\sigma(y')$ and
  $\lca_T(x,y')\prec_T\lca_T(x,y)$ since $y$ must be a best match of $x$ by
  Definition~\ref{def:qbm}(i)). The latter remains true upon
  restriction of $T$ to $L'$. Because of Equation~\eqref{eq:restriction-u} and
  $y\in N(x,s)\setminus \{v\}$, we have $u'(x,s)=u(x,s)$ or
  $u'(x,s)=v^*$. In the former case, we obtain
  $\lca_{T'}(x,y)=\lca_T(x,y)\preceq_T u(x,s)=u'(x,s)$, implying
  $y\in N'(x,s)$. In the latter case, i.e., if $u'(x,s)=v^*$, we have
  $u(x,s)=w\ne \lca_{T}(x,y)$.  Thus, we obtain
  $\lca_{T}(x,y)\preceq_{T} v^*$ from $\lca_T(x,y)\preceq_T u(x,s)$ and we
  observe that $\lca_T(x,y)\preceq_T v$ is ruled out by $y\in N_s(x)$.
  Therefore, $y\in N'(x,s)$.  If $y\notin N(x,s)$ for some $y$ of color
  $s$, then $u(x,s)\prec_{T} \lca_{T}(x,y)$ or there is some $y'\in L$ of
  color $s$ such that $\lca_T(x,y')\prec_{T}\lca_T(x,y)$ and in particular
  $y'\in N(x,s)$.  If $u(x,s)\prec_{T} \lca_{T}(x,y)$, then by construction
  $u'(x,s)\prec_{T'}\lca_{T'}(x,y)$.  In the latter case, $y'\in N(x,s)$
  implies $y'\ne v$, and thus, $y'\in L'$ and
  $\lca_{T'}(x,y')\prec_{T'}\lca_{T'}(x,y)$.  In both cases, we, therefore,
  have $y\notin N'(x,s)$. In summary, $v\notin N(x,s)$ implies
  $N(x,s)=N'(x,s)$ for all $x\in L'$.
  
  Since $v\in N(x,s)$ and hence $s=\sigma(v)$, it is useful to distinguish
  two cases: Case (i): $N(x,s)\setminus\{v\}\ne \emptyset$. Let
  $q\coloneqq \lca_T(x,v)$. If $v\notin\child_T(q)$ or $|\child_T(q)|>2$,
  then $q\in V(T')$ and hence $\lca_T(x,y)=\lca_{T'}(x,y)$ for all
  $y\in L'$ and $u'(x,s)=u(x,s)$. This implies
  $N'(x,s)= N(x,s)\setminus\{v\}$. If $\child_T(q)=\{v,v^*\}$, then
  $\lca_T(x,y)=q$ implies $y=v$. Thus $v\in N(x,s)$ implies $N(x,s)=\{v\}$,
  contradicting $N(x,s)\setminus\{v\}\ne\emptyset$. Case (ii): If
  $N(x,s)=\{v\}$, then $N'(x,s)=\emptyset$ because $u'(x,s)=x$. In summary,
  $N'(x,s)=N(x,s)\setminus\{v\}$ holds for all $x\in L'$ and $s\in S$. This
  implies that $(G',\sigma')$ is the subgraph of $(G,\sigma)$ induced by
  $L'=L\setminus\{v\}$.
\end{proof}
Since the removal of any set of vertices can be carried out step by step by
deleting one vertex a time, it turns out that the induced subgraph
$(G,\sigma)[L']$ has an explanation in terms of the restriction $T_{L'}$ of
$T$ to the leaf set $L'\subseteq L$. Therefore, the following result is
obtained.
\begin{corollary}
  \label{cor:hereditary}
  Every induced subgraph of a qBMG is a qBMG, i.e., the qBMGs form a
  hereditary graph class.
\end{corollary}

\begin{fact}\label{fact:components}
  The disjoint union $(G,\sigma)=\bigcupdot_{i=1}^m (G_i,\sigma_i)$ is a
  qBMG if and only if each of $(G_i,\sigma_i)$ is a qBMG.
\end{fact}
\begin{proof}
  Since each connected component of $G$ is an induced subgraph the
  \emph{only-if}-direction is trivial. Conversely, consider explanations
  $(T_i,\sigma_i,u_i)$ for the disjoint graphs $(G_i,\sigma_i)$. Let $T$ be
  the tree whose vertices are those of $T_i$ together with a new root
  $\rho$ added as the common parent of the roots $\rho_i$ of $T_i$. It has
  the leave set $L(T)=\bigcupdot_i L(T_i)$.  Define
  $\sigma(x)\coloneqq \sigma_i(x)$ and $u_T(x,r)=u_{T_i}(x,r)$ for
  $x\in L_i$. There are no edges joining vertices from different
  components since $u(x,r)\preceq_T \rho_{T_i}\prec_T \rho$ while for each
  connected component, we have $G[L(T_i)]=G_i$. Thus
  $(G,\sigma)=\qbmg(T,\sigma,u)$.
\end{proof}
The corresponding result for BMGs requires the additional condition that
all connected components use the same set of colors, i.e.
$\sigma(V(G_i))=\sigma(V(G_j))$, see~\cite[]{Geiss:19a}.

\section{Recognition of qBMGs and Rooted Triples}
\label{sect:triples}

The recognition problem for BMGs, and more generally for qBMGs is of
practical interest as part of workflows for orthology detection. The
graph $(G,\sigma)$ recording for each query gene the most similar
sequences in each target genome can be computed 
efficiently~\cite{HernandezSalmeron:20} and serves as an empirical 
approximation for a
BMG or qBMG. The empirical estimate $(G,\sigma)$ may contain false
positive and false negative edges; thus, it is, in general, neither a BMG
nor a qBMG. Solving the recognition problem is the first key step toward
the more difficult problem of identifying potential errors in the
input data. For BMGs, the recognition problem has been solved 
in~\cite{Schaller:21b}, see Proposition~\ref{prop:BMG-triple-charac} in terms 
of two conditions: (i) the absence of color-sinks and (ii) consistency of a
collection of informative and forbidden triples.

We have already seen that qBMGs are not color-sink-free in general,
i.e., they may violate condition~(i), see Figure~\ref{fig:qBMG-example}. In
contrast, triple consistency, i.e., condition~(ii), remains valid as shown
by the following lemma, which generalizes Lemmas~2.11 and~3.2
of~\cite{Schaller:21b} from BMGs to qBMGs:
\begin{lemma}
  \label{lem:qBMG-inf-forb-triples}
  Let $(G,\sigma)$ be a qBMG explained by $(T,\sigma,u)$. Then $T$ displays
  all triples in $\mathscr{R}(G,\sigma)$ but none of the triples in
  $\mathscr{F}(G,\sigma)$.  In particular, 
  $(\mathscr{R}(G,\sigma),\mathscr{F}(G,\sigma))$ is consistent.
\end{lemma}
\begin{proof}
  First, suppose that $ab|b'\in \mathscr{R}(G,\sigma)$, i.e., $ab\in E(G)$
  and $ab'\notin E(G)$.  Since $ab\in E(G)$, there is no $b''$ of color
  $\sigma(b'')=\sigma(b)(=\sigma(b'))$ such that
  $\lca_T(a,b'')\prec_T\lca_T(a,b)$ and that
  $\lca_T(a,b)\preceq_T u(a,\sigma(b))$.  In particular,
  $\lca_T(a,b)\preceq_T\lca_T(a,b')$. On the other hand, if
  $\lca_T(a,b')=\lca_T(a,b)\preceq_T u(a,\sigma(b))$, then $ab'\in E(G)$
  since $(T,\sigma,u)$ explains the qBMG $(G,\sigma)$.  Therefore,
  $\lca_T(a,b)\prec_T \lca_T(a,b')$ is the only remaining possibility and
  $T$ displays the triple $ab|b'$.
  
  Now suppose that $ab|b'\in \mathscr{F}(G,\sigma)$, i.e., $ab\in E(G)$ and
  $ab'\in E(G)$.  By similar arguments as above, this implies
  $\lca_T(a,b)\preceq_T\lca_T(a,b')$ and
  $\lca_T(a,b')\preceq_T\lca_T(a,b)$, respectively.  Therefore,
  $\lca_T(a,b)=\lca_T(a,b')$ and hence $T$ does not display the triple
  $ab|b'$. Similarly, $T$ does not display $ab'|b$ either.
\end{proof}

\begin{theorem}
  \label{thm:qBMG-via-triples}
  A properly-colored digraph $(G,\sigma)$ with vertex set $L$ is a
  qBMG if and only if $(\mathscr{R}(G,\sigma),\mathscr{F}(G,\sigma))$
  is consistent.  In this case, for every tree $T$ on $L$ that agrees with
  $(\mathscr{R}(G,\sigma),\mathscr{F}(G,\sigma))$, there is a truncation
  map $u$ such that $(T,\sigma,u)$ explains $(G,\sigma)$.
\end{theorem}
\begin{proof}
  If $(G,\sigma)$ is a qBMG, then consistency of
  $(\mathscr{R}(G,\sigma),\mathscr{F}(G,\sigma))$ follows from
  Lemma~\ref{lem:qBMG-inf-forb-triples}.  Now suppose that
  $(\mathscr{R}(G,\sigma),\mathscr{F}(G,\sigma))$ is consistent.  From
  $L_{\mathscr{R}(G,\sigma)\cup\mathscr{F}(G,\sigma)}\subseteq L$, at least
  one tree with leaf set $L$ displays all triples in
  $\mathscr{R}(G,\sigma)$ and none of the triples in
  $\mathscr{F}(G,\sigma)$. Let $(T,\sigma)$ be any tree with this
  property. Set $S\coloneqq \sigma(L)$ and consider the truncation
  map $u\colon L \times S \to V(T)$ that is given by $u(x,s)=x$ if
  $N_G(x,s)=\emptyset$, and $u(x, s)=\rho_T$ otherwise, for all
  $s\in S$ and $x\in L$. Note that $u$ is a well-defined truncation map
  for $(T,\sigma)$, since $(G,\sigma)$ is properly colored and hence
  $u(x,\sigma(x))=x$ for all $x\in L$.
  
  Let $(\widehat{G},\sigma)\coloneqq \qbmg(T,\sigma,u)$ be the qBMG
  explained by $(T,\sigma,u)$, $x\in L$ and $s\in S$. It remains to show
  that $(\widehat{G},\sigma) =(G,\sigma)$.  For $N_G(x,s)=\emptyset$, we
  have set $u(x,s)=x$ and thus $ N_{\hat G}(x,s)=\emptyset$ as a
  consequence of condition~(ii) in Definition~\ref{def:qbm}.  Now suppose
  $N_G(x,s)\ne\emptyset$. This is possible only if $\sigma(x)\ne s$.  By
  construction, we have $u(x,s)=\rho_T$, and therefore,
  $y\in N_{\hat G}(x,s)$ with $\sigma(y)=s$ if and only if $y$ is a best
  match of $x$ in $(T,\sigma)$.  Consider a vertex $y\in N_G(x,s)$ and
  assume, for contradiction, that $y$ is not a best match of $x$ in
  $(T,\sigma)$, i.e., there is some $y'$ of color $s$ such that
  $\lca_T(x,y')\prec_T\lca_T(x,y)$.  If $y'\in N_G(x,s)$, then
  $xy'|y \in \mathscr{F}(G,\sigma)$.  If $y'\notin N_G(x,s)$, then
  $xy|y' \in \mathscr{R}(G,\sigma)$.  In both cases, the agreement of $T$
  with $(\mathscr{R}(G,\sigma),\mathscr{F}(G,\sigma))$ contradicts that
  $\lca_T(x,y')\prec_T\lca_T(x,y)$.  Hence, $y$ must be a best match of $x$
  and thus $y\in N_{\hat G}(x,s)$.  Now assume $y\notin N_G(x,s)$.  In this
  case, $N_G(x,s)\ne\emptyset$ implies the existence of some
  $y'\in N_G(x,s)$ of color $s$ and distinct from $y$.  Hence,
  $xy'|y\in \mathscr{R}(G,\sigma)$ is displayed by $T$ and thus,
  $\lca_T(x,y')\prec_T\lca_T(x,y)$.  Therefore, $y$ is not a best match of
  $x$ in $(T,\sigma)$, and $y\notin N_{\hat G}(x,s)$.  In summary, we have
  $N_{\hat G}(x,s)=N_{G}(x,s)$ for all $x\in L$ and all
  $s\in S\setminus\{\sigma(x)\}$.  Therefore, we conclude that
  $(G,\sigma)=(\widehat{G},\sigma)$ and that $(T,\sigma,u)$ explains the
  qBMG $(G,\sigma)$.
\end{proof}

As a direct consequence of Proposition~\ref{prop:BMG-triple-charac} and
Theorem~\ref{thm:qBMG-via-triples} we obtain
\begin{theorem}\label{thm:NewChar}
  A properly colored graph $(G,\sigma)$ is a BMG if and only
  if $(G,\sigma)$ is a color-sink-free qBMG.
\end{theorem}

The proof of Theorem~\ref{thm:qBMG-via-triples} is constructive and thus
provides an algorithm to decide computationally whether $(G,\sigma)$ is
qBMG and, if so, to compute an explanation $(T,\sigma,u)$ for $(G,\sigma)$.
The procedure is summarized in Algorithm~\ref{alg:qBMG}. It
relies on the polynomial-time algorithm \texttt{MTT}~\cite{He:06}, named
for the ``mixed triplets problem restricted to trees'', which decides
whether $(\mathscr{R},\mathscr{F})$ is consistent and -- in the
affirmative case -- constructs a corresponding tree $T$.  \texttt{MTT} in
turn can be understood as a generalization of the well-known
\texttt{BUILD} algorithm~\cite{Aho:81}. Given a set of rooted triples
$\mathscr{R}$ defined on a set of leaves $L$, \texttt{BUILD} produces an
undirected auxiliary graph, called \emph{Aho et al.\ graph} and denoted by
$[\mathscr{R},L]$, with vertex set $L$ and edges $xy$ if and only if there
is some $z\in L$ such that $xy|z\in \mathscr{R}$. \texttt{BUILD} then
recurses on the connected components of $[\mathscr{R}, L]$ with singleton
vertex sets serving as base cases. The algorithm returns a tree $T$ (on
$L$) displaying all triples in $\mathscr{R}$, which is determined by the
recursion hierarchy and denoted by $\build(\mathscr{R},L)$, or fails if no
such tree exists. The latter is the case if and only if, at some recursion
step with $|L'|>1$, the Aho et al.\ graph $[\mathscr{R}',L']$ is connected.

\begin{algorithm}[tb]
  \small 
  \caption{\texttt{qBMG recognition}}
  \label{alg:qBMG}
  \begin{algorithmic}[1]
    \Require  A  vertex-colored digraph $(G,\sigma)$ with
    $\sigma\colon V(G)\to S$
    \Ensure   A tree $(T,\sigma,u)$ that explains $(G,\sigma)$ if
    it is a qBMG and, otherwise, \texttt{false}
    \If{$(G,\sigma)$ is not properly colored}
    \State \Return \texttt{false}
    \EndIf
    \State Compute $\mathscr{R}(G,\sigma)$ and $\mathscr{F}(G,\sigma)$
    according to Definition~\ref{def:informative_triples}
    \State Use \texttt{MTT} to check if
    $(\mathscr{R}(G,\sigma), \mathscr{F}(G,\sigma))$
    is consistent and, in the positive case, compute an agreeing
    tree $T$;\newline  otherwise \Return \texttt{false}
    \State Initialize $u(x,s)=x$ for all $(x,s)\in L\times S$
    \ForAll{$xy\in E(G)$}
    \State $u(x,\sigma(y))\leftarrow\rho_T$
    \EndFor
    \State 	\Return $(T,\sigma,u)$
  \end{algorithmic}
\end{algorithm}

\begin{corollary}\label{cor:runtime_qBMGrecognition}
  Algorithm~\ref{alg:qBMG} with input $(G=(L,E),\sigma)$ can be
  implemented to run in $O(|E| |L|^2\log |L|)$ and decides whether
  $(G,\sigma)$ is a qBMG and, in the affirmative case, constructs a tree
  $(T,\sigma,u)$ that explains $(G,\sigma)$.
\end{corollary}
\begin{proof}
  It takes $O(|E|)$ time to verify whether $(G,\sigma)$ is a properly
  colored digraph.  The triple sets $\mathscr{R}(G,\sigma)$ and
  $\mathscr{F}(G,\sigma)$ may be obtained in $O(|L|\,|E|)$ time since
  every triple in $\mathscr{R}(G,\sigma) \cup \mathscr{F}(G,\sigma)$
  is identifiable by an edge $e$ and a vertex not incident with $e$. Given a 
  pair $(\mathscr{R}, \mathscr{F})$ of triple sets defined
  on $L$, the algorithm \texttt{MTT} decides in $O(|\mathscr{R}||L| +
  |\mathscr{F}| |L| \log |L| + |L|^2 \log |L|)$ time whether $(\mathscr{R},
  \mathscr{F})$ is consistent and, if so, returns a corresponding tree $T$
  with the same time complexity~\cite{He:06}. Since $|\mathscr{R}|\le
  |L|\,|E|$ and $|\mathscr{F}|\le |L|\,|E|$, we obtain an upper bound of
  $O(|E|\,|L|^2\log |L|)$ for \texttt{MTT}.  Finally, the truncation
  map $u$ is constructed in $O(|L||S| + |E|)$ by first initializing $u(x,s)=x$
  for all $(x,s)\in L\times S$.  Then the edges are visited in
  arbitrary order. We set  $u(x,s)\coloneqq\rho_T$ if there is an edge 
  $e=xy$ with $\sigma(y)=s$.
  Since only colors in $\sigma(L)$ are considered, we may assume
  $|S|\le |L|$. The total effort is therefore dominated by \texttt{MTT}.
\end{proof}

The following technical result shows that the informative and forbidden
triples in a subgraph $(G,\sigma)[V']$ induced by $V'\subseteq V(G)$ are
exactly the respective sets of triples of the original graph restricted to
$V'$.
\begin{fact}[\cite{Schaller:21c}, Observation~2]
  \label{obs:R-restriction}
  Let $(G,\sigma)$ be a vertex-colored digraph and $V'\subseteq V(G)$.
  Then $R(G,\sigma)_{V'}=R(G[V'],\sigma_{V'})$ holds for each
  $R\in\{\mathscr{R},\mathscr{F},\Rbin\}$.
\end{fact}
Theorem~\ref{thm:qBMG-via-triples} and Observation~\ref{obs:R-restriction} yield
an alternative proof for the fact that qBMGs form a hereditary graph class.
\medskip
\par\noindent\emph{Alternative Proof of Corollary~\ref{cor:hereditary}.}
Let $(G,\sigma)$ be a qBMG. Then
$(\mathscr{R}\coloneqq\mathscr{R}(G,\sigma),
\mathscr{F}\coloneqq\mathscr{F}(G,\sigma))$ is consistent by
Theorem~\ref{thm:qBMG-via-triples}. By Observation~\ref{obs:R-restriction},
we have
$\mathscr{R'}\coloneqq\mathscr{R}(G[V'],\sigma_{V'})\subseteq\mathscr{R}$
and
$\mathscr{F'}\coloneqq\mathscr{F}(G[V'], \sigma_{V'})\subseteq\mathscr{F}$
for every $V'\subseteq V(G)$. Therefore, the pair
$(\mathscr{R'}, \mathscr{F'})$ is clearly still consistent. By
Theorem~\ref{thm:qBMG-via-triples}, $(G,\sigma)[V']$ is therefore again a
qBMG.  \qed \medskip

\begin{figure}[h]
  \centering
  \includegraphics[width=0.8\linewidth]{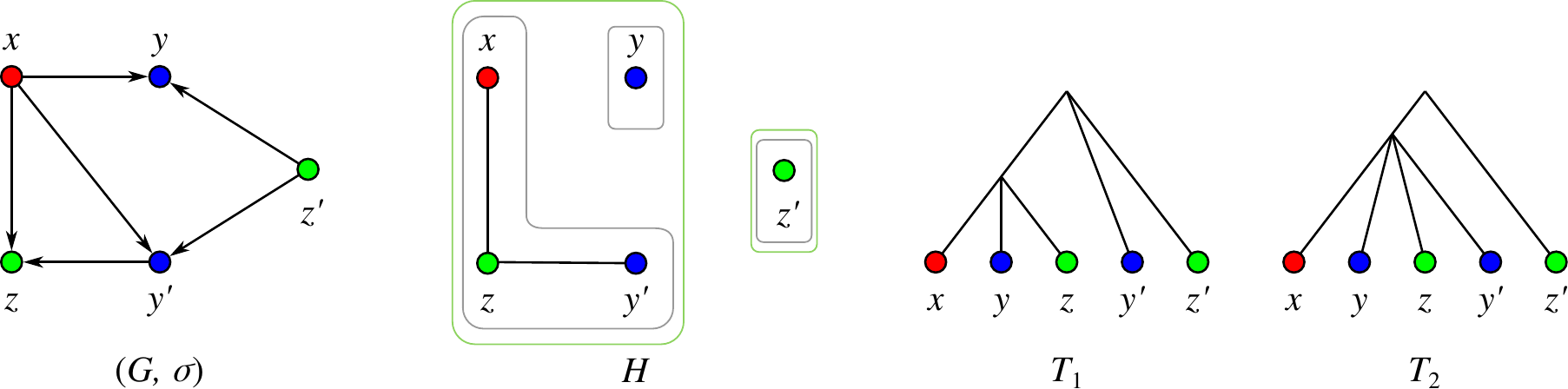}
  \caption{Example for a qBMG $(G,\sigma)$ that is not explained by
    $(T_1=\build(\mathscr{R}(G,\sigma), V(G)), \sigma, u)$ for any
    truncation map $u$. The set of informative triples is
    $\mathscr{R}(G,\sigma)=\{xz|z', y'z|z'\}$. It yields the Aho et
    al.\ graph $H$ in the top-level recursion step of
    \texttt{BUILD}. The gray frames indicate the connected components
    of $H$ and the green boxes indicate the sets in the auxiliary
    partition produced by the \texttt{MTT} algorithm. The tree $T_1$
    cannot explain the edge $xy'$ for any truncation map because
    $\sigma(y)=\sigma(y')$ and $\lca_{T_1}(x,y)\prec \lca_{T_1}(x,y')$. The
    tree produced by \texttt{MTT} with input $\mathscr{R}$ and
    $\mathscr{F}$ is $T_2$.}
  \label{fig:BUILD-fails}
\end{figure}

In~\cite{Schaller:21d}, we obtained a characterization of BMGs that uses
only the informative triples and does not explicitly utilize the set of
forbidden triples. More precisely, a vertex-colored digraph $(G,\sigma)$
with vertex set $L$ is a BMG if and only if (i) the set of informative
triples $\mathscr{R}\coloneqq \mathscr{R}(G,\sigma)$ is consistent, and
(ii) $(G,\sigma) = \bmg(\build(\mathscr{R},L), 
\sigma)$~\cite[Theorem~1]{Schaller:21d}. This provides a procedure to recognize
BMGs that is different from the one based on 
Proposition~\ref{prop:BMG-triple-charac}.

However, the example in Figure~\ref{fig:BUILD-fails} shows that an
analogous result does not hold for qBMGs. The counterexample $(G,\sigma)$
has informative triples $\mathscr{R}(G,\sigma)=\{xz|z', y'z|z'\}$ and
forbidden triples $\mathscr{F}(G,\sigma)=\{xy|y', xy'|y, z'y|y',
z'y'|y\}$. The Aho et al.\ graph $H$ constructed at the top-level recursion
step of \texttt{BUILD} has three connected components. The final output of
\texttt{BUILD} is the tree $T_1$. Since
$\lca_{T_1}(x,y)\prec_{T_1}\lca_{T_1}(x,y')$, the edge $xy'$ can never be
contained in the qBMG explained by $(T_1,\sigma,u)$ for any truncation map
$u$. Hence, condition~(ii) does not hold.  The \texttt{MTT} algorithm also
constructs the Aho et al.\ graph in each recursion step but merges two
components $C$ and $C'$ whenever there is a forbidden triple $ab|c$ such
that $a,b\in C$ and $c\in C'$. This yields an \emph{auxiliary partition}
$\mathcal{D}(L)$ of the leaf set $L$ whose sets serve as input for the
recursive calls instead of the connected components of the Aho et
al.\ graph. In the example, \texttt{MTT} merges the two components
$\{x,y',z\}$ and $\{y\}$ in the top-level recursion step in response to the
forbidden triple $xy'|y\in \mathscr{F}(G,\sigma)$.  The final result is the
tree $T_2$. Also, one can easily verify that a truncation map $u'$ can be
found for $T_2$ such that $(T_2,\sigma,u')$ explains $(G,\sigma)$.  Later
we will use the following property of the trees produced by algorithm
\texttt{MTT}.
\begin{lemma}
  \label{lem:MTT-minimal}
  Let $(\mathscr{R},\mathscr{F})$ be a consistent pair of two triple sets
  defined on leaf set $L$.  Then \texttt{MTT} returns a least resolved tree
  $T$ on $L$ that agrees with $(\mathscr{R},\mathscr{F})$, i.e., there is
  no tree $T'$ on $L$ with $T'<T$ that still agrees with
  $(\mathscr{R},\mathscr{F})$.
\end{lemma}
\begin{proof}
  Since $(\mathscr{R},\mathscr{F})$ is consistent, \texttt{MTT} returns a
  tree $T$ on $L$ that agrees with $(\mathscr{R},\mathscr{F})$, see 
  ~\cite[Theorem~1]{He:06}.  To prove that $T$ is least resolved, we
  show first that every inner edge in $T$ is distinguished by some triple
  in $\mathscr{R}$.  Assume, for contradiction, that there is an inner edge
  $vw\in E(T)$ that is not distinguished by a triple in $\mathscr{R}$.
  Since $vw$ is an inner edge, $w$ has children $w_1,\dots,w_k$, $k\ge 2$.
  Consider the recursion step of \texttt{MTT} on $L'\coloneqq L(T(v))$ and
  $(\mathscr{R}_{L'},\mathscr{F}_{L'})$.  The algorithm constructs an
  auxiliary partition $\mathcal{D}(L')$ by starting with the connected
  components of $[\mathscr{R}_{L'}, L']$ and then merging two components
  $C$ and $C'$ stepwisely as long as there is a forbidden triple
  $ab|c\in \mathscr{F'}$ such that $a,b\in C$ and $c\in C'$. Note that
  $L(T(w_1))\cupdot\dots\cupdot L(T(w_k))= L(T(w))\in \mathcal{D}(L')$.
  There cannot be a triple $ab|c\in\mathscr{R}_{L'}$ with $a\in L(T(w_i))$
  and $b\in L(T(w_j))$ for distinct children $w_i$ and $w_j$ of $w$.  To
  see this, consider that $c\in L'\setminus L(T(w))$.  In this case,
  $\lca_{T}(a,b)=w$ and $\lca_{T}(\{a,b,c\})=v$, i.e., the triple $ab|c$
  would distinguish the edge $vw$.  If, on the other hand, $c\in L(T(w))$,
  then $T$ clearly cannot display the triple $ab|c\in \mathscr{R}_{L'}
  \subseteq \mathscr{R}$ also as a consequence of $\lca_{T}(a,b)=w$.
  Hence, no such triple exists.  Therefore, no two vertices $a\in
  L(T(w_i))$ and $b\in L(T(w_j))$ for distinct children $w_i$ and $w_j$ of
  $w$ are adjacent in $[\mathscr{R}_{L'}, L']$.  It follows that since $w$
  has at least two children and $L(T(w))\in \mathcal{D}(L')$, the set
  $L(T(w))$ must have been emerged as the disjoint union of $l\ge k\ge 2$
  connected components $C_1,\dots,C_l$ of $[\mathscr{R}_{L'}, L']$ in
  response to forbidden triples in $\mathscr{F}_{L'}$.  In particular,
  $C_i\subseteq L(T(w_j))$ for some $1\le j\le k$ for each $1\le i\le k$.
  Consider the series of merging steps that involve sets $C_1,\dots,C_l$
  and unions of these sets. There must be a first merging step of
  two sets $C$ and $C'$ in this series that satisfy $C\subseteq L(T(w_i))$
  and $C'\subseteq L(T(w_j))$ for distinct $w_i,w_j\in \child_T(w)$ in
  response to some triple $ab|c\in \mathscr{F}_{L'}$ such that $a,b\in C$
  and $c\in C'$.  This implies $\lca_{T}(a,b)\preceq_{T}w_i \prec_{T} w
  =\lca_{T}(\{a,b,c\})$, and thus $T$ displays $ab|c\in
  \mathscr{F}_{L'}\subseteq \mathscr{F}$; a contradiction.  Hence, every
  inner edge in $T$ is distinguished by some triple in $\mathscr{R}$.
  
  Let $T'$ be a tree obtained from $T$ by a non-empty series of edge
  contractions, say of edges $e_1,\dots, e_m$. By the arguments above,
  $e_1$ is distinguished by a triple $ab|c\in \mathscr{R}$.  Therefore, we
  have $\lca_{T_{e_1}}(a,b)=\lca_{T_{e_1}}(\{a,b,c\})$ after contraction of
  $e_1$, and thus, $ab|c$ is not displayed by $T_{e_1}$.  Since further
  contraction of edges does not introduce newly displayed triples, 
  see~\cite[Theorem~1]{Bryant:95}. It follows that $T'$ does not display
  $\mathscr{R}$ and hence $T$ is least resolved.
\end{proof}

Finally, we point out that increasing the number of colors by splitting a
color class preserves the qBMG property.
\begin{proposition}
  If $(G,\sigma)$ is a $\ell$-qBMG with $\sigma(L)=\ell<|L|$ colors, then
  there is a proper coloring $\sigma'$ such that $(G,\sigma')$ is an
  $(\ell+1)$-qBMG.
  \label{prop:numcol}
\end{proposition}
\begin{proof}
  Write $L_s\coloneqq\{x\in L\mid\sigma(x)=s\}$. If $\ell<|L|$, there is a
  color $s$ with $|L_s|\ge2$. Let $\sigma'$ be a coloring of $G$ obtained
  by arbitrarily partitioning $L_s=L_{s'}\cupdot L_{s''}$ into two
  non-empty subsets with new colors $s'$ and $s''$.  Let $(T,\sigma,u)$ be
  an explanation of $(G,\sigma)$. For the leaf-colored tree $(T,\sigma')$,
  we construct the truncation map $u'$ as follows: For $x\in L_s$, we set
  $u'(x,s')=u'(x,s'')=x$.  For $x\notin L_s$, we set $u'(x,s')=x$ if
  $N_s(x)\cap L_{s'}=\emptyset$ and $u'(x,s'')=x$ if
  $N_s(x)\cap L_{s''}=\emptyset$; otherwise $u'(x,s')=u(x,s)$ and
  $u'(x,s'')=u(x,s)$, respectively. Finally, for $x\in V$ and
  $t\notin\{s',s''\}$, we set $u'(x,t)=u(x,t)$. It is not difficult to
  check that, by construction, $N_G(x,s)=N_G(x,s')\cup N_G(x,s'')$ for all
  $x\in V$, while the out-neighborhoods of all vertices and all other
  colors remain unchanged. Thus $(G,\sigma')=\qbmg(T,\sigma',u')$.  Since
  $\sigma'$ is a proper $(\ell+1)$-coloring, $(G,\sigma')$ is an
  $(\ell+1)-qBMG$.
\end{proof}

\section{Least Resolved Trees for qBMGs}
\label{sec:lrt}

In general, many trees explain a given qBMG $(G,\sigma)$. Among these
trees, the least resolved ones are of particular interest since they
describe the phylogenetic information implicit in $(G,\sigma)$ without
adding internal vertices and, thus, evolutionary events that are not
implied by the available best matches. \emph{Least-resolved trees} (LRTs)
thus are the most parsimonious explanations.  It is of particular
interest whether there is a unique LRT and thus unambiguous phylogenetic
information or whether there are conflicting explanations in the form of
mutually inconsistent trees.  BMGs are explained by a unique LRT
$\widehat{T}$~\cite{Geiss:19a}. As we shall see below, this is no
longer true for qBMGs.  Although we lose uniqueness, the LRTs of qBMGs
still have some convenient properties. In the following paragraphs, we
briefly summarize the situation.

\begin{definition}
  Let $(T,\sigma,u)$ be a leaf-colored tree with truncation map
  $u$. An edge $e\in E(T)$ is \emph{redundant} with respect to
  the explained qBMGs if there is a truncation map $u'$ such that
  $\qbmg(T_e,\sigma,u')=\qbmg(T,\sigma,u)$, i.e., $(T,\sigma,u)$ and
  $(T_e,\sigma,u')$ explain the same qBMG.  An edge $e\in E(T)$ that
  is not redundant is \emph{essential} with respect to the
  explained qBMGs.  Moreover, $(T,\sigma,u)$ is \emph{least resolved} if
  there is no tree $T'<T$ and truncation map $u'$ such that
  $(T',\sigma,u')$ explains $\qbmg(T,\sigma,u)$.
\end{definition}

\begin{figure}[tb]
  \centering
  \includegraphics[width=0.5\linewidth]{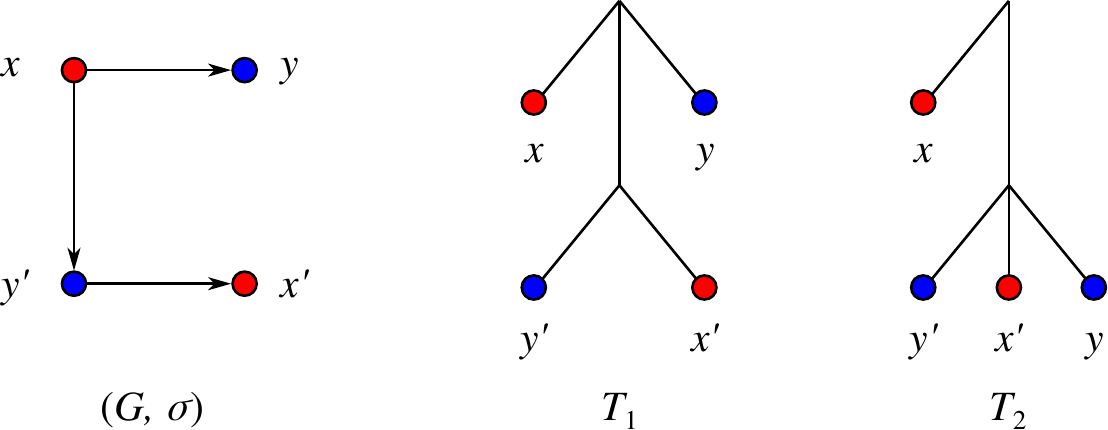}
  \caption{A connected 2-colored qBMG $(G,\sigma)$ with least resolved
    trees (with truncation maps) $(T_1,\sigma,u_1)$ and
    $(T_2,\sigma,u_2)$. For $i\in\{1,2\}$, the truncation maps can be
    chosen such that $u_i(v, r)=\rho_{T_i}$ for all $v\in V(G)$ and
    $r\in S\setminus\{\sigma(v)\}$ if $N(x,r)\ne\emptyset$, and
    $u_i(v, r)=v$ otherwise.}
  \label{fig:LRTs-not-unique}
\end{figure}

\begin{lemma}
  \label{lem:essential-edges}
  Let $(G,\sigma)$ be a qBMG explained by $(T,\sigma,u)$ and let
  $e=vw\in E(T)$ with $w\prec_T v$. Then $e$ is essential if and only if
  $w\in L(T)$ or there are $x,y,y'\in L(T)$ such that $w=\lca_T(x,y)$,
  $xy\in E(G)$, $v=\lca_T(x,y')$ and $\sigma(y')=\sigma(y)$.
\end{lemma}
\begin{proof}
  We start with the \emph{if}-direction.  Clearly, if $w\in L(T)$, then
  $L(T)\ne L(T_e)$ and thus, there is no truncation map $u$ such that
  $(T_e,\sigma,u')$ still explains $(G,\sigma)$.  Now suppose there are
  $x,y\in L(T)$ such that $w=\lca_T(x,y)$ and $xy\in E(G)$, and
  $y'\in L(T)$ with $\lca_T(x,y')=v$ and $\sigma(y')=\sigma(y)$.  Since
  $\lca_T(x,y)=w\prec_T v=\lca_T(x,y')$, we have $xy'\notin E(G)$.  After
  contraction of $e$, we have $\lca_{T_e}(x,y)=\lca_{T_e}(x,y')$.  Hence,
  we have $xy\in E(\qbmg(T_e,\sigma,u'))$ if and only if
  $xy'\in E(\qbmg(T_e,\sigma,u'))$ for any truncation map $u'$. Therefore,
  $(G,\sigma)$ cannot be explained by any tree of the form
  $(T_e,\sigma,u')$ and thus $e$ is essential.
  
  We continue with the \emph{only-if}-direction.  To this end suppose, for
  contraposition, that $e=vw$ is an inner edge and that there are no
  $x,y, y'\in L(T)$ such that $w=\lca_T(x,y)$, $xy\in E(G)$,
  $v=\lca_T(x,y')$ and $\sigma(y')=\sigma(y)$. By
  Lemma~\ref{lem:qBMG-inf-forb-triples}, $T$ agrees with
  $(\mathscr{R}(G,\sigma), \mathscr{F}(G,\sigma))$.  Since $T_e<T$, every
  triple that is displayed by $T_e$ is also displayed by $T$, 
  see~\cite[Theorem~6.4.1]{Semple:03}. Therefore, $T_e$ does not display any
  of the forbidden triples in $\mathscr{F}(G,\sigma)$.  Now suppose that
  $T_e$ does not display some triple $ab|b'\in\mathscr{R}(G,\sigma)$. By
  definition, we have $ab\in E(G)$, $ab'\notin E(G)$ and
  $\sigma(a)\ne\sigma(b)=\sigma(b')$.  By
  Lemma~\ref{lem:qBMG-inf-forb-triples}, $T$ displays $ab|b'$, i.e.,
  $\lca_{T}(a,b)\prec_{T} \lca_{T}(a,b')=\lca_{T}(b,b')$.  Since $T$ and
  $T_e$ differ only by contraction of $e$ and $T_e$ does not display
  $ab|b'$, it must hold that $\lca_{T}(a,b)=w$ and $\lca_{T}(a,b')=v$.
  Together with $\sigma(b)=\sigma(b')$ and $ab\in E(G)$, this contradicts
  the assumption. Hence, $T_e$ displays all triples in
  $\mathscr{R}(G,\sigma)$.  In summary, $T_e$ agrees with
  $(\mathscr{R}(G,\sigma), \mathscr{F}(G,\sigma))$. By
  Theorem~\ref{thm:qBMG-via-triples}, there is a truncation map $u'$ such
  that $(T_e,\sigma,u')$ explains the qBMG $(G,\sigma)$. Hence, $e$ is
  redundant and thus not essential.
\end{proof}
The conditions $w=\lca_T(x,y)$, $v=\lca_T(x,y')$, and $\sigma(y')=\sigma(y)$
in Lemma~\ref{lem:essential-edges} imply that $xy'\notin E(G)$. Together
with $xy\in E(G)$, we obtain
\begin{corollary}
  Let $(G,\sigma)$ be a qBMG explained by $(T,\sigma,u)$. An inner edge $vw\in 
  E(T)$ is essential if and only if it is distinguished by an informative 
  triple $xy|y'\in\mathscr{R}(G,\sigma)$.
\end{corollary}
The following result shows that it is possible to characterize least
resolved trees for qBMGs by the absence of redundant edges:
\begin{lemma}
  \label{lem:LRT}
  A tree $(T,\sigma,u)$ is least resolved if and only if it does not
  contain a redundant edge.
\end{lemma}
\begin{proof}
  Let $(G,\sigma)$ be the qBMG explained by $(T,\sigma,u)$.
  If $e\in E(T)$ is redundant, there is a truncation map $u'$ such that
  $(T_e,\sigma,u')$ explains $(G,\sigma)$. Since $T_e<T$,
  $(T,\sigma,u)$ is not least resolved.
  
  For the converse, suppose $(T,\sigma,u)$ does not contain redundant edges
  and, for contradiction, assume that there is a tree $T'<T$ and a
  truncation map $u'$ such that $(T',\sigma,u')$ explains $(G,\sigma)$.
  Therefore, we have $L(T)=L(T')$, and thus, $T'$ is obtained from $T$ by a
  series of at inner-edge contractions $e_1,\dots, e_k$ with $k\ge 2$ since
  otherwise $e_1$ would be redundant by definition.  Since $e_1=vw$ with
  $w\prec_T v$ is essential, we have, by Lemma~\ref{lem:essential-edges},
  $x,y\in L(T)$ such that $w=\lca_T(x,y)$ and $y\in N_{G}(x,\sigma(y))$ and
  $y'\in L(T)$ with $\lca_T(x,y')=v$ and $\sigma(y')=\sigma(y)$.  In
  particular, therefore, $y'\notin N_{G}(x,\sigma(y))$.  After contraction
  of $e_1$ in $T$, we have $\lca_{T_{e_1}}(x,y)=\lca_{T_{e_1}}(x,y')$.
  Hence, there is no set $A\in \mathcal{H}(T_{e_1})$ such that (a)
  $x,y\in A$ and $y'\notin A$ or (b) $x,y'\in A$ and $y\notin A$.  Since
  $T'<T_{e_1}$ and $L(T)=L(T_{e_1})$, we have
  $\mathcal{H}(T')\subset\mathcal{H}(T_{e_1})$.  Therefore, there is also
  no set $A\in \mathcal{H}(T')$ such that (a) or (b) holds.  It follows
  that $\lca_{T'}(x,y)=\lca_{T'}(x,y')$.  Since $(T',\sigma,u')$ explains
  $(G,\sigma)$, this immediately implies that either
  $y,y'\in N_{G}(x,\sigma(y))$ or $y,y'\notin N_{G}(x,\sigma(y))$; a
  contradiction.
\end{proof}

Even though Lemma~\ref{lem:LRT} generalizes the situation in BMGs, we no
longer have uniqueness of the least resolved
tree. Figure~\ref{fig:LRTs-not-unique} gives a simple counterexample with
only two colors. Nevertheless, as a consequence of
Theorem~\ref{thm:qBMG-via-triples} and Lemma~\ref{lem:MTT-minimal}, we
obtain
\begin{proposition}
  \label{prop:MTT-qBMG-least-resolved}
  Let $(G,\sigma)$ be a qBMG with vertex set $L$. Then $(T,\sigma,u)$, where 
  $T$ is the tree produced by algorithm \texttt{MTT} with input 
  $(\mathscr{R}(G,\sigma), \mathscr{F}(G,\sigma))$ and $L$ and $u$ is a 
  suitable truncation map is least resolved and explains $(G,\sigma)$.
\end{proposition}

\section{Binary-explainable qBMGs}
\label{sect:be-qBMG}

Binary Trees play a prominent role in phylogenetics because
evolutionary events are usually assumed to give rise to only two
descendant genes or species. Polytomies, i.e., vertices in $T$ with three
or more children, are, therefore, in most cases interpreted as the
consequence of insufficient phylogenetic information (soft polytomies)
rather than as true multifurcations (hard polytomies). The latter appear 
occasionally but are considered very rare by most 
authors~\cite{Gorecki:14,Sayyari:18}. It is of interest, therefore, to determine
whether a qBMG can be explained by a binary tree $T$. BMGs that can be
explained by binary trees have been studied in~\cite{Schaller:21c}.  In
this section, we derive analogous results for qBMGs.

\begin{definition}
  A vertex-colored digraph $(G,\sigma)$ is a binary-explainable qBMG if there 
  is a binary tree $T$ and a truncation map $u$ such that $(T,\sigma,u)$ 
  explains $(G,\sigma)$.
\end{definition}

\begin{theorem}
  \label{thm:be-qBMG}
  A properly-colored digraph $(G,\sigma)$ with vertex set $L$ is a
  binary-explainable qBMG if and only if $\Rbin\coloneqq \Rbin(G,\sigma)$
  is consistent.
  In this case, for every refinement $T$ of the tree $\build(\Rbin, L)$, there 
  is a truncation map $u$ such that $(T,\sigma,u)$ explains $(G,\sigma)$.
\end{theorem}
\begin{proof}
  For the \emph{only-if} direction, suppose $(G,\sigma)$ is a
  binary-explainable qBMG, i.e., it is explained by a binary leaf-colored
  tree with truncation map $(T,\sigma,u)$.  By
  Theorem~\ref{thm:qBMG-via-triples}, $T$ agrees with
  $(\mathscr{R}(G,\sigma), \mathscr{F}(G,\sigma))$.  In particular, $T$
  displays all triples in $\mathscr{R}(G,\sigma)\subseteq \Rbin$.  Now
  assume that there is a triple $bb'|a\in \Rbin \setminus
  \mathscr{R}(G,\sigma) = \{bb'|a\colon ab|b'\in \mathscr{F}(G,\sigma),
  \sigma(b)=\sigma(b')\}$.  By definition, $ab|b'\in \mathscr{F}(G,\sigma)$
  implies $ab'|b\in \mathscr{F}(G,\sigma)$.  Therefore, of the three
  possible triples $ab|b'$, $ab'|b$, and $bb'|a$ on $\{a,b,b'\}$, only
  $bb'|a$ may be displayed by $T$. This, together with the hypothesis
  that $T$ is binary, implies that $T$ indeed displays $bb'|a$.  Since this
  is true for any $bb'|a\in \Rbin \setminus \mathscr{R}(G,\sigma)$, $T$
  displays all triples in $\Rbin$, and thus, $\Rbin$ is consistent.
  
  For the \emph{if} direction, suppose that $\Rbin$ is consistent.  Hence,
  the tree $\build(\Rbin, L)$ exists and displays all triples in $\Rbin$.
  Consider an arbitrary refinement $T$ of $\build(\Rbin, L)$ (note that
  $T=\build(\Rbin, L)$ is possible).  By~\cite[Theorem~1]{Bryant:95}, $T$
  also displays all triples in $\Rbin$.  Since
  $\mathscr{R}(G,\sigma)\subseteq \Rbin$, $T$ displays all triples in
  $\mathscr{R}(G,\sigma)$.  Now assume there is a triple
  $ab|b'\in \mathscr{F}(G,\sigma)$ where $\sigma(b)=\sigma(b')$.  By
  definition, we have $bb'|a\in \Rbin$, and thus, $bb'|a$ is displayed by
  $T$. Therefore, $T$ does not display the triple $ab|b'$.  Since
  this is true for any $ab|b'\in \mathscr{F}(G,\sigma)$, $T$ displays none
  of the triples in $\mathscr{F}(G,\sigma)$.  In summary, $T$ agrees with
  $(\mathscr{R}(G,\sigma), \mathscr{F}(G,\sigma))$.  This yields that
  $(\mathscr{R}(G,\sigma), \mathscr{F}(G,\sigma))$ is consistent and,
  together with Theorem~\ref{thm:qBMG-via-triples}, that $(G,\sigma)$ is a
  qBMG.  In particular, $T$ can be chosen to be a binary refinement of
  $\build(\Rbin, L)$. Hence, $(G,\sigma)$ is a binary-explainable qBMG.
\end{proof}
As an immediate consequence of Theorem~\ref{thm:be-qBMG},
Observation~\ref{obs:R-restriction}, and the fact that subsets of consistent
triple sets are again consistent, we obtain that binary-explainable qBMGs
form a hereditary class of colored digraphs.
\begin{corollary}
  \label{cor:be-qBMG-hereditary}
  Every induced subgraph of a binary-explainable qBMG is again a
  binary-explainable qBMG.
\end{corollary}

In~\cite{Schaller:21a}, it was furthermore shown that a simple forbidden
induced subgraph, called \emph{hourglass}, is sufficient to characterize
binary-explainable BMGs among BMGs in general.
\begin{definition}
  An \emph{hourglass} in a properly vertex-colored digraph $(G,\sigma)$,
  denoted by $[xy \hourglass x'y']$, is a subgraph $(G[Q],\sigma_{Q})$
  induced by a set of four pairwise distinct vertices
  $Q=\{x, x', y, y'\}\subseteq V(G)$ such that (i)
  $\sigma(x)=\sigma(x')\ne\sigma(y)=\sigma(y')$, (ii) $xy,yx,x'y',y'x'\in 
  E(G)$, (iii) $xy',yx'\in E(G)$, and (iv)
  $y'x,x'y\notin E(G)$.\\
  A properly vertex-colored digraph is \emph{hourglass-free} if it does not 
  have an hourglass as an induced subgraph.
\end{definition}
\noindent 
The definition of hourglasses is illustrated in Figure~\ref{fig:hourglass}
(leftmost digraph).  We will use the following technical result to link
hourglasses to the inconsistency of the triple set $\Rbin(G,\sigma)$.
\begin{proposition}[{\cite[Theorem~2]{Bryant:95}}]
  \label{prop:consistency-aho-graph}
  A set of triples $\mathscr{R}$ defined on a leaf set $L$ is consistent if and 
  only if $[\mathscr{R}_{L'}, L']$ is disconnected for every subset 
  $L'\subseteq L$ with $|L'|\ge 3$.
\end{proposition}

\begin{figure}[tb]
  \centering
  \includegraphics[width=0.8\linewidth]{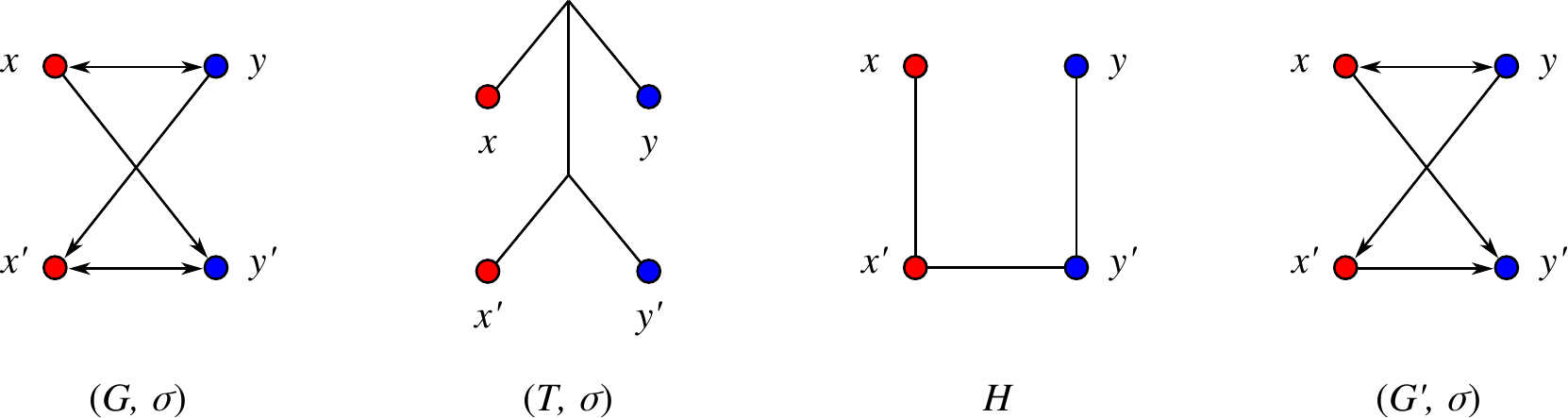}
  \caption{An hourglass $(G,\sigma)$ is itself a BMG (and thus qBMG) since
    it is explained by the tree $(T,\sigma)$.  The Aho et al.\ graph
    $H=[\Rbin(G,\sigma), V(G)]$ is connected, and thus, $(G,\sigma)$ is not
    binary-explainable, see Theorem~\ref{thm:be-qBMG} and
    Proposition~\ref{prop:consistency-aho-graph}). Since
    $H=[\Rbin(G',\sigma), V(G')]$, the qBMG $(G',\sigma)$ is also not
    binary-explainable.}
  \label{fig:hourglass}
\end{figure}

\begin{lemma}
  \label{lem:be-implies-hourglass-free}
  Every binary-explainable qBMG $(G,\sigma)$ is hourglass-free.
\end{lemma}
\begin{proof}
  Suppose that $(G,\sigma)$ has an hourglass $[xy \hourglass x'y']$ as a
  subgraph induced by $L'\coloneqq \{x,x',y,y'\}$ (where
  $\sigma(x)=\sigma(x')\ne\sigma(y)=\sigma(y')$).  By the definition of
  hourglasses, informative triples, and forbidden triples, we have
  $x'y'|y, y'x'|x \in \mathscr{R}(G,\sigma)$ and
  $xy|y',xy'|y,yx|x', yx'|x\in\mathscr{F}(G,\sigma)$.  This in turn yields
  $x'y'|y, y'x'|x, yy'|x, xx'|y \in \Rbin(G,\sigma)\eqqcolon\Rbin$.  Hence,
  $[\Rbin_{L'}, L']$ has (undirected) edges $x'y'$, $yy'$, and $xx'$, and
  thus, it is connected.  By Proposition~\ref{prop:consistency-aho-graph},
  this implies that $\Rbin$ is not consistent, and by
  Theorem~\ref{thm:be-qBMG}, $(G,\sigma)$ is not a binary-explainable qBMG.
\end{proof}
In contrast to BMGs, see~\cite[Proposition~8]{Schaller:21a}, the
converse of Lemma~\ref{lem:be-implies-hourglass-free} is not true for
qBMGs. Figure~\ref{fig:hourglass} shows an example of an hourglass-free
qBMG $(G',\sigma)$ that is not binary-explainable.
\begin{figure}[tb]
  \centering
  \includegraphics[width=0.85\linewidth]{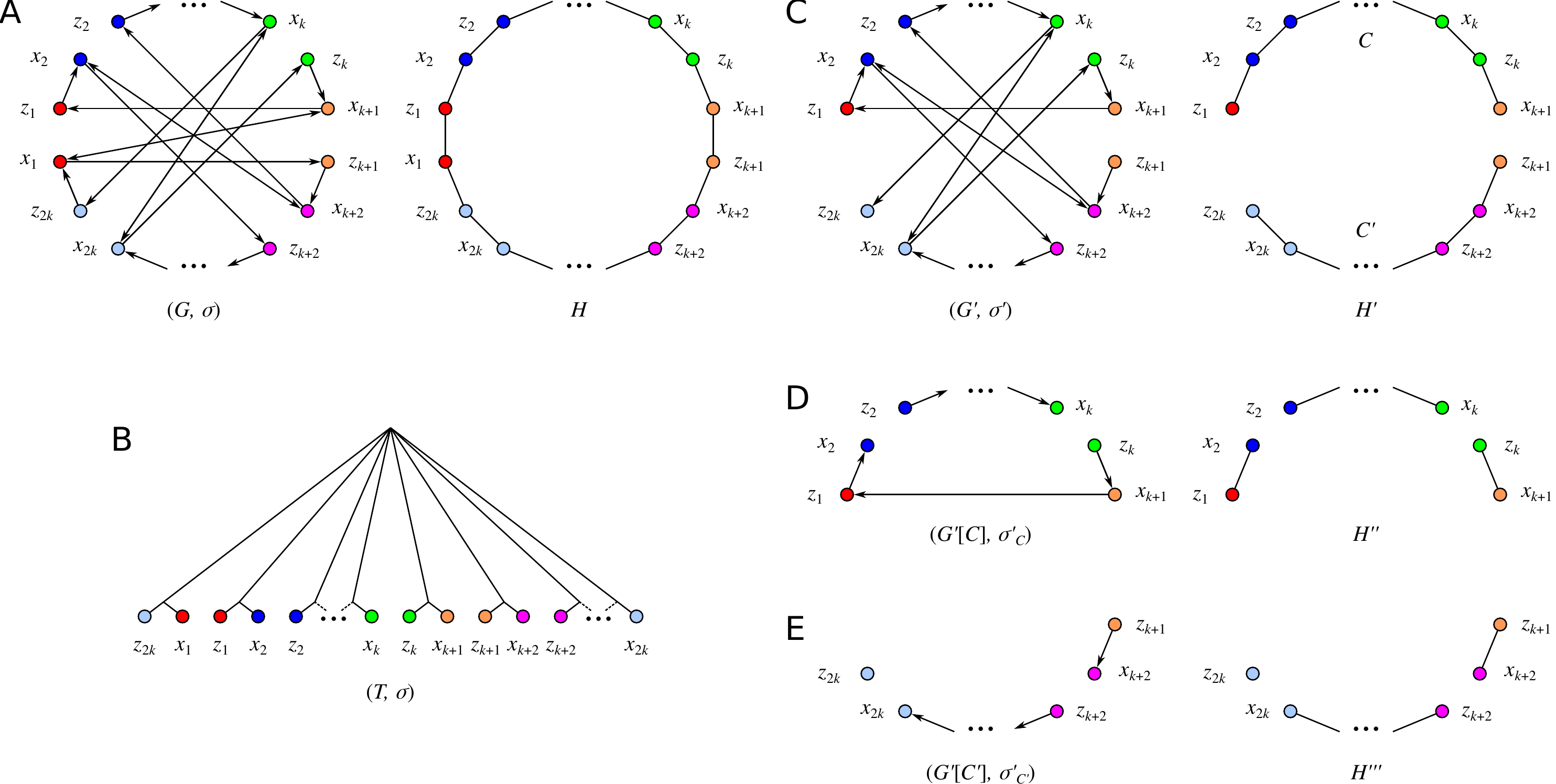}
  \caption{The vertex-colored digraph $(G,\sigma)$ is a qBMG since it is
    associated with the BMG explained by $(T,\sigma)$. However, it is not
    binary-explainable since $H=[\Rbin(G, \sigma), V(G)]$ is connected and
    thus $\Rbin(G, \sigma)$ is inconsistent, see
    Theorem~\ref{thm:be-qBMG} and Proposition~\ref{prop:consistency-aho-graph}. 
    The graph $(G',\sigma')$ obtained from $(G,\sigma)$ by deleting vertex 
    $x_1$ and its incident edges is a binary-explainable qBMG. To see this, 
    note that $H'=[\Rbin(G', \sigma'), V(G')]$ is disconnected. Hence, 
    \texttt{BUILD} with input $\Rbin(G', \sigma')$ and $V(G')$ recurses on the 
    connected components $C$ and $C'$. In these recursion steps, the Aho 
    et al.\ graphs $H''$ and $H'''$, respectively, are disconnected and their
    connected components have at most two vertices. Hence, $\Rbin(G',
    \sigma')$ is consistent.}
  \label{fig:large-non-be-qBMG}
\end{figure}
In general, no finite set of forbidden induced subgraphs characterizes
binary-explainable qBMGs among qBMGs. Figure~\ref{fig:large-non-be-qBMG}A, 
arising from a geometric configuration, provides a counterexample. For an 
integer $k\ge 2$, consider a regular $4k$-gon $x_1z_1\cdots x_{2k}z_{2k}$ in 
the Euclidean plane. Its vertices are colored with $2k$ colors such that $x_i$ 
and $z_i$ receive the same color for all $1\leq i\leq 2k$. Then, we insert
directed edges $x_{1}z_{k+1}$, $x_{k+1}z_1$, $z_1x_2$ and the symmetric
edge $x_1x_{k+1}$. Corresponding arcs are inserted for each of the $2k-1$
rotation of the polygon by an angle of $\pi i/k$ for $1 \leq i\leq k-1$
around its center; see Figure~\ref{fig:large-non-be-qBMG}A. A formal
presentation of this graph is given in the following lemma, in which we use 
brackets to denote indices that are taken modulo $2k$.

\begin{lemma}
  \label{lem:be-qBMG-not-finite}
  For an integer $k \geq 2,$ let $(G,\sigma)$ be a vertex-colored
  digraph with vertex set $L$ and edge set $E$ where
  \begin{description}[noitemsep,nolistsep]
    \item[(i)] $L=\{x_1,z_1, x_2, z_2, \dots, x_{2k},
    z_{2k}\}$,
    \item[(ii)] $\sigma(L)\coloneqq \{c_1,c_2,\dots,c_{2k}\}$
    comprises $2k$ pairwise distinct colors such that
    $\sigma(x_i)=\sigma(z_i)=c_i$ for all $1\le i\le 2k$, and
    \item[(iii)] $E= \bigcup_{i=1}^{2k} \{ x_{i} x_{\MOD{i+k}},\; x_{i}
    z_{\MOD{i+k}},\; z_{i} x_{\MOD{i+1}}\}$
    where $[n]\coloneqq n\ \mathrm{mod}\ 2k$.
  \end{description}
  Then $(G,\sigma)$ is a qBMG that is not binary-explainable.  Moreover,
  every induced subgraph of $(G,\sigma)$ with at most $|L|-2$ vertices is a
  binary-explainable qBMG.
\end{lemma}
\begin{proof}
  By construction, for all $i\in \{1,\dots, 2k\}$, we have
  $z_{i} x_{\MOD{i+1}}\in E$ but $z_{i} z_{\MOD{i+1}}\not\in E$.  This
  and $\sigma(z_i)\neq \sigma(x_{\MOD{i+1}}) =\sigma(z_{\MOD{i+1}})$
  implies $z_{i}x_{\MOD{i+1}}|z_{\MOD{i+1}}\in \mathscr{R}(G,\sigma)$ for
  all $i\in \{1,\dots, 2k\}$. Therefore,
  \begin{equation*}
    \mathscr{R}\coloneqq 
    \{z_{i}x_{\MOD{i+1}}|z_{\MOD{i+1}} \colon 1\le i \le 
    2k\}\subseteq\mathscr{R}(G,\sigma).
  \end{equation*}
  Assume, for contradiction, that
  $\mathscr{R}(G,\sigma)\setminus \mathscr{R}\neq \emptyset$, i.e.,
  $\mathscr{R}(G,\sigma)$ contains further informative triples $ab|b'$ in
  which case $ab\in E(G)$, $ab'\notin E(G)$ and
  $\sigma(a)\neq \sigma(b)=\sigma(b')$.  By construction, $ab\in E(G)$
  implies that only $a = x_{i}$ and either $b=x_{\MOD{i+k}}$ or
  $b= z_{\MOD{i+k}}$ is possible for some $i\in \{1,\dots, 2k\}$.  In
  either case, $\sigma(b)=\sigma(b')$ implies that
  $b'\in \{x_{\MOD{i+k}},z_{\MOD{i+k}}\}\setminus \{b\}$ and thus
  $ab'\in E(G)$; a contradiction.  Hence, there are no other informative
  triples, i.e., $\mathscr{R}=\mathscr{R}(G,\sigma)$.  Similar arguments
  imply
  \begin{equation*}
    \mathscr{F}\coloneqq \mathscr{F}(G,\sigma)=\bigcup_{i=1}^{2k}\; 
    \{x_{i} x_{\MOD{i+k}}|z_{\MOD{i+k}},\; x_{i} 
    z_{\MOD{i+k}}|x_{\MOD{i+k}}\}
  \end{equation*}
  and thus
  \begin{equation*}
    \Rbin\coloneqq \Rbin(G,\sigma)=\bigcup_{i=1}^{2k}\; 
    \{z_{i}x_{\MOD{i+1}}|z_{\MOD{i+1}},\;
    x_{\MOD{i+k}}z_{\MOD{i+k}} | x_{i}\}.
  \end{equation*}
  Using the tree $(T,\sigma)$ in Figure~\ref{fig:large-non-be-qBMG}B, one
  easily verifies that this tree displays all triples in $\mathscr{R}$ and
  none of the triples in $\mathscr{F}$.  Therefore,
  $(\mathscr{R},\mathscr{F})$ is consistent, and by
  Theorem~\ref{thm:qBMG-via-triples}, $(G,\sigma)$ is a qBMG.  However, the
  graph $H\coloneqq[\Rbin, L]$ is connected (see also
  Figure~\ref{fig:large-non-be-qBMG}A), and thus
  Proposition~\ref{prop:consistency-aho-graph} implies that $\Rbin$ is not
  consistent. By Theorem~\ref{thm:be-qBMG}, this implies that $(G,\sigma)$
  is not binary-explainable.
  
  We next show that $(G,\sigma)[L\setminus\{x_i\}]$ is a binary-explainable
  qBMG for every $1\le i\le 2k$.  By the symmetric roles of $x_i$ in
  $(G,\sigma)$, it suffices to show this claim for $x_1$.  Thus consider
  $L'=L\setminus\{x_1\}$ and $(G',\sigma')\coloneqq (G,\sigma)[L']$. First
  note that the induced subgraph $(G',\sigma')$ is again a qBMG by 
  Corollary~\ref{cor:hereditary}.  By Observation~\ref{obs:R-restriction}, we 
  have
  \begin{equation*}
    \Rbin(G',\sigma')=\Rbin(G,\sigma)_{L'}
    =\left(\bigcup_{i=1}^{2k}\; 
    \{z_{i}x_{\MOD{i+1}}|z_{\MOD{i+1}},\;
    x_{\MOD{i+k}}z_{\MOD{i+k}} | x_{i}\}\right)
    \setminus \{z_{2k} x_1 | z_1, \; x_{k+1} z_{k+1} | x_1\}.
  \end{equation*}
  Since no second triple of the form $x_{k+1} z_{k+1} | y$ exists, the
  edge $x_{k+1} z_{k+1}$ is not present in the Aho et al.\ graph
  $H'=[\Rbin(G',\sigma'), L']$.  In particular, $H'$ is disconnected (see
  Figure~\ref{fig:large-non-be-qBMG}C) and has two connected components
  $C=\{z_1, x_2, z_2\dots, x_k, z_k, x_{k+1}\}$ and
  $C'=\{z_{k+1}, x_{k+2}, z_{k+2}, x_{2k}, z_{2k}\}$.  One easily
  verifies that there is no triple of the form
  $x_{\MOD{i+k}}z_{\MOD{i+k}} | x_{i}$ in $\Rbin(G',\sigma')$
  that is contained in $\Rbin(G',\sigma')_{C}$ or $\Rbin(G',\sigma')_{C'}$.
  In particular, this implies that $\Rbin(G',\sigma')_{C}$ and
  $\Rbin(G',\sigma')_{C'}$ contain only informative triples of
  $(G',\sigma')$.  Since $(G',\sigma')$ is a qBMG,
  $\mathscr{R}(G',\sigma')$ is consistent by
  Theorem~\ref{thm:qBMG-via-triples}. The latter arguments together imply
  that \texttt{BUILD} with input $\Rbin(G',\sigma')$ and $L'$ never
  encounters a connected Aho et al.\ graph with more than two vertices,
  i.e., $\Rbin(G',\sigma')$ is consistent.  By Theorem~\ref{thm:be-qBMG},
  this implies that $(G',\sigma')$ is binary-explainable.
  
  Now consider an induced subgraph
  $(G'',\sigma'')\coloneqq (G,\sigma)[L'']$ where
  $L''=L\setminus\{y,\tilde{y}\}$ for two distinct vertices
  $y,\tilde{y}\in L$.  If there is some $1\le i\le 2k$ such that $y=x_i$ or
  $\tilde{y}=x_i$, then $(G'',\sigma'')$ is the induced subgraph of a
  binary-explainable qBMG $(G,\sigma)[L\setminus \{x_i\}]$ and thus also a
  binary-explainable qBMG by Corollary~\ref{cor:be-qBMG-hereditary}.  Now
  suppose that this is not the case, i.e., $y=z_i$ or $\tilde{y}=z_j$ and
  $1\le i < j \le 2k$.  By construction of the Aho et al.\ graph
  and since $\Rbin(G'',\sigma'')=\Rbin_{L''}$ by
  Observation~\ref{obs:R-restriction}, $[\Rbin(G'',\sigma''), L'']$ is a 
  subgraph
  of $H[L'']$.  Using Figure~\ref{fig:large-non-be-qBMG}A, we therefore
  observe that $H[L'']$ and thus also $[\Rbin(G'',\sigma''), L'']$ has at
  least two connected components.  In particular, \texttt{BUILD} with input
  $\Rbin(G'',\sigma'')$ and $L''$ recurses on these connected components, and 
  none of them can contain both $x_i$ and $x_j$. Let $C$ be such a
  connected component and suppose that $x_i\notin C$.  Recall
  that \texttt{BUILD} recurses on $\Rbin(G'',\sigma'')_{C}$ and $C$.  By
  Observation~\ref{obs:R-restriction}, we have
  $\Rbin(G'',\sigma'')_{C}=\Rbin(G''[C],\sigma''_C)$.  Since
  $(G''[C],\sigma''_C)$ is an induced subgraph of a binary-explainable qBMG
  $(G,\sigma)[L\setminus \{x_i\}]$, it is also a binary-explainable qBMG by 
  Corollary~\ref{cor:be-qBMG-hereditary}.  It follows that
  $\Rbin(G'',\sigma'')_{C}$ must be consistent.  Since this is true for all
  connected components in the top-level recursion step, \texttt{BUILD}
  never encounters a connected Aho et al.\ graph with more than one vertex,
  and thus, $\Rbin(G'',\sigma'')$ is consistent. By
  Theorem~\ref{thm:be-qBMG}, this implies that $(G'',\sigma'')$ is
  binary-explainable.
  
  Finally an induced subgraph $(G^*,\sigma^*)$ of $(G,\sigma)$ with at most 
  $|L|-2$ 
  vertices is in particular, also an induced subgraph of a graph
  $(G,\sigma)[L'']$ with $|L''|=|L|-2$.
  By the arguments above, $(G,\sigma)[L'']$ is a binary-explainable qBMG.
  By Corollary~\ref{cor:be-qBMG-hereditary}, $(G^*,\sigma^*)$ is also a 
  binary-explainable qBMG.
\end{proof}
Since $k\ge2$ in Lemma~\ref{lem:be-qBMG-not-finite} can be chosen
arbitrarily large, we conclude
\begin{corollary}
  \label{cor:be-qBMG-not-finite}
  There is no finite set of forbidden induced vertex-colored subgraphs that 
  characterize the subclass of binary-explainable qBMGs among the class of 
  qBMGs.
\end{corollary}
In particular, there are minimal forbidden induced subgraphs with an 
arbitrarily large number of colors.

\section{Two-Colored qBMGs}
\label{sect:2qBMG}

The restrictions of qBMGs to the vertices with two colors in a sense form
``building blocks'' for the more general case. These bipartite graphs
deserve a more detailed investigation. With only two colors, say
$\sigma(L)=\{r,s\}$, we have $N(x)=N(x,s)$ if $\sigma(x)=r$ and
$N(x)=N(x,r)$ if $\sigma(x)=s$. We can, therefore, largely omit explicit
references to the vertex colors $r$ and $s$ and focus entirely on the sets
$N(x)$ and $N^-(x)$. In particular, properly 2-colored digraphs are
color-sink-free precisely if they are sink-free.  To avoid the explicit
treatment of trivial cases, we will consider the monochromatic edge-less
graph, and in particular singleton graphs $K_1$, also as 2-qBMGs.

In~\cite{Schaller:21d} 2-BMGs are characterized as sink-free graphs
satisfying (N1), (N2), and (N3), and in Theorem~\ref{thm:NewChar} we have
identified BMGs as the sink-free qBMGs. This suggests but does not imply,
that 2-qBMGs are characterized by the neighborhood conditions (N1), (N2),
and (N3). In the following, we show that this conjecture is indeed true.

\begin{lemma}\label{lem:n1n2n3_quasi}
  Every 2-qBMG $(G,\sigma)$ satisfies \AX{(N1)}, \AX{(N2)}, and \AX{(N3)}.
\end{lemma}
\begin{proof}
  Let $(\tilde{G},\sigma)$ be a 2-BMG associated with the qBMG $(G,\sigma)$
  and $L$ the common vertex set of $(G,\sigma)$ and
  $(\tilde{G},\sigma)$. By Proposition~\ref{prop:char2-BMG} and
  Lemma~\ref{lem:N1N2N3-old-eq-new}, $(\tilde{G},\sigma)$ satisfies
  \AX{(N1)}, \AX{(N2)}, and \AX{(N3)}.
  
  In order to show that $(G,\sigma)$ satisfies \AX{(N1)}, we first consider
  a pair of independent vertices $x$ and $y$ in $\tilde{G}$.  Since
  $G\subseteq \tilde G$, $x$ and $y$ are also independent in $G$. Using
  $E(G)\subseteq E(\tilde{G})$, the conclusion of property \AX{(N1)}
  remains true for $G$ whenever $x$ and $y$ are independent in $\tilde G$.
  Now consider two vertices $x$ and $y$ independent in $G$ but not
  in $\tilde G$ and assume that $xy\in E(\tilde{G})$. Thus,
  $N_{G}(x)\neq N_{\tilde G}(x)$ and Lemma~\ref{2out_neighbour} implies
  that $N_{G}(x) = \emptyset$.  Hence, $(G,\sigma)$ trivially satisfies
  \AX{(N1)} since there cannot be a vertex $t$ with $xt\in E(G)$.
  
  To see that $(G,\sigma)$ satisfies \AX{(N2)} assume that
  $x_1 y_1, y_1 x_2, x_2 y_2 \in E(G)$.  Recall that
  $E(G)\subseteq E(\tilde G)$ and $(\tilde G,\sigma)$ satisfies \AX{(N2)}
  and thus, $x_1 y_2 \notin E(\tilde G)$.  Hence, if $x_1 y_2 \notin E(G)$,
  we have $N_{G}(x)\neq N_{\tilde G}(x)$ and Lemma~\ref{2out_neighbour}
  implies that $N_{G}(x) = \emptyset$, contradicting the assumption
  $x_1 y_1 \in E(G)$. Thus we have $x_1 y_2 \in E(G)$, and hence
  $(G,\sigma)$ satisfies \AX{(N2)}.
  
  It remains to show that $(G,\sigma)$ satisfies \AX{(N3)}.  Consider
  $x,y\in L$ with a common out-neighbor $z$ in $G$.  By
  Lemma~\ref{2out_neighbour}, this implies that $N_{G}(x)=N_{\tilde G}(x)$
  and $N_{G}(y)=N_{\tilde G}(y)$. In particular, $z$ is also a common
  out-neighbor of $x$ and $y$ in $\tilde G$ since $G\subseteq \tilde G$.
  Now assume, for contradiction, that neither $N_{G}(x)\subseteq N_{G}(y)$
  nor $N_{G}(y)\subseteq N_{G}(x)$.  Hence, there are vertices $v$ such
  that $xv\in E(G)$ and $yv\notin E(G)$, and $w$ such that $yw\in E(G)$ and
  $xw\notin E(G)$.  Since $E(G)\subseteq E(\tilde{G})$, we have
  $v\in N_{\tilde G}(x)$ and $w\in N_{\tilde G}(y)$.  Since
  $(\tilde{G},\sigma)$ satisfies \AX{(N3)} and $x$ and $y$ have a common
  out-neighbor $z$ in $\tilde{G}$, it follows that
  $N_{\tilde G}(x) \subseteq N_{\tilde G}(y)$ or
  $N_{\tilde G}(y) \subseteq N_{\tilde G}(x)$.  If
  $N_{\tilde G}(x) \subseteq N_{\tilde G}(y)$, we have
  $v\in N_{\tilde G}(x) \subseteq N_{\tilde G}(y) = N_{G}(y)$, i.e.,
  $yv\in E(G)$; a contradiction.  If
  $N_{\tilde G}(y) \subseteq N_{\tilde G}(x)$, we similarly obtain the
  contradiction that $xw\in E(G)$.  Therefore, we conclude that
  $(G,\sigma)$ satisfies \AX{(N3)}.
\end{proof}

The hierarchy-like structure of the out-neighborhood, i.e., property
\AX{(N3)}, suggests that the out-neighborhoods contain information on the
structure on the tree(s) explaining a 2-colored qBMG.  This connection,
however, is less straightforward than one might expect.  
Following~\cite{Geiss:19a}, we consider the \emph{reachable sets}
\begin{equation}
  R(x)\coloneqq N(x) \cup N(N(x)) \cup N(N(N(x))) \cup \cdots
\end{equation}
and the corresponding isotonic map $R:2^V\to 2^V$,
$A\mapsto R(A)\coloneqq \bigcup_{x\in A} R(x)$. As a direct consequence of
\AX{(N2)}, we have $R(x) = N(x)\cup N(N(x))$.  Furthermore, we observe
$N(R(A))=N(N(A))\cup N(N(N(A))) \subseteq N(N(A))\cup N(A)=R(A)$ and thus
also $N(N(R(A)))\subseteq N(R(A))\subseteq R(A)$, which implies
$R(R(A))\subseteq R(A)$.

\begin{lemma}
  \label{lem:N-empty-NN-empty}
  Let $(G,\sigma)$ be a properly 2-colored graph satisfying \AX{(N1)} and
  \AX{(N2)}. If $N(x)\cap N(y)=\emptyset$, then
  $N(N(x))\cap N(N(y))=\emptyset$.
\end{lemma}
\begin{proof}
  Assume, for contradiction, that $N(x)\cap N(y)=\emptyset$ and there is a
  vertex $w\in N(N(x))\cap N(N(y))$. Then neither $x$ nor $y$ is a sink,
  $x$ and $y$ are distinct, and, since $(G,\sigma)$ is properly colored, we
  have $\sigma(x)=\sigma(y)=\sigma(w)$.  Moreover, there are vertices
  $u\in N(x)$ and $v\in N(y)$ such that $xu, uw, yv, vw \in E(G)$.  In
  particular, $u\ne v$ and $xv,yu\notin E(G)$ because
  $N(x)\cap N(y)=\emptyset$.  Therefore, $x,y,u$, and $v$ must all be
  pairwise distinct.  If $x=w$, then $yv, vw=vx, xu\in E(G)$ and
  $yu\notin E(G)$ contradicts \AX{(N2)}. The case $y=w$ yields an analogous
  contradiction.  Since $(G,\sigma)$ is properly colored, it remains to
  consider the case when all five vertices $x$, $y$, $u$, $v$, and $w$ are
  all pairwise distinct.  Suppose that $x$ and $v$ are
  independent. Together with $xu,vw,uw$ this contradicts \AX{(N1)}.  Hence,
  we must have $vx\in E(G)$.  But then $yv,vx,xu\in E(G)$ and \AX{(N2)}
  imply $yu\in E(G)$; a contradiction.  Therefore, we conclude that a
  vertex $w\in N(N(x))\cap N(N(y))$ cannot exist.
\end{proof}

We are now in the position to generalize~\cite[Lemma~9]{Geiss:19a}, which is 
equivalent to the next statement if one assumes $(G,\sigma)$ to be connected.
\begin{lemma}
  \label{lem:L9}
  Let $(G,\sigma)$ be a properly 2-colored graph satisfying \AX{(N1)},
  \AX{(N2)}, and \AX{(N3)}. Then the set of reachable sets
  $\mathcal{R}\coloneqq \left\{R(x)\mid x \in V(G)\right\}$ forms a
  hierarchy-like set system.
\end{lemma}
\begin{proof}
  Let $x,y\in V(G)$. If $x=y$, we trivially have $R(x)\cap
  R(y)=R(x)=R(y)$. If $x$ or $y$ is a sink, we have
  $R(x)\cap R(y)=\emptyset$. Now assume that $R(x)\cap R(y)\ne
  \emptyset$. By Lemma~\ref{lem:N-empty-NN-empty}, this implies either (i)
  $N(x)\cap N(y)\ne \emptyset$, or (ii) $N(x)\cap N(y)=\emptyset$ and
  $N(x)\cap N(N(y))\ne \emptyset$ or $N(N(x))\cap N(y)\ne \emptyset$.
  Since $(G,\sigma)$ is properly colored, Case~(i) implies
  $\sigma(x)=\sigma(y)$ while Case~(ii) can only occur if
  $\sigma(x)\ne\sigma(y)$.  In Case~(i), \AX{(N3)} implies
  $N(x)\subseteq N(y)$ or $N(y)\subseteq N(x)$, which by isotony implies
  $R(x)\subseteq R(y)$ or $R(y)\subseteq R(x)$, respectively.
  
  For Case~(ii), assume $\sigma(x)\ne\sigma(y)$ and
  $N(x)\cap N(N(y))\ne \emptyset$ and let $v\in N(x)\cap N(N(y))$. From
  $\sigma(x)\ne\sigma(y)$, we infer $\sigma(v)=\sigma(y)$.  If $x\in R(y)$,
  then isotony of the map $R$ and the arguments above imply
  $R(x)\subseteq R(R(y))\subseteq R(y)$.  Similarly, $y\in R(x)$ implies
  $R(y)\subseteq R(x)$.  Now assume $x\notin R(y)$ and $y\notin R(x)$ and
  thus in particular $x\notin N(y)$ and $y\notin N(x)$.  Together with
  $v\in N(x)\cap N(N(y))$, this implies that $x,y,v$ are pairwise distinct.
  Moreover, there must be a vertex $w\in N(y)$ with $yw, wv\in E(G)$ and
  $w\notin\{x,y,v\}$. Thus, we have $xv,yw,wv\in E(G)$ and $x$ and $y$ are
  independent vertices $x$ and $y$; a contradiction to \AX{(N1)}.
  Analogous arguments show that only $R(x)\subseteq R(y)$ or
  $R(y)\subseteq R(x)$ are possible if $N(N(x))\cap N(y)\ne \emptyset$.
  Therefore, we have $R(x)\cap R(y)\in\{\emptyset, R(x), R(y)\}$ for all
  $x,y\in V$.
\end{proof}
We note that the proof of Lemma~\ref{lem:L9} in~\cite{Geiss:19a} relies on
the assumption that $(G,\sigma)$ is sink-free~\cite{Schaller:21d}. The
proof above shows that this additional assumption is not
necessary.

As in the special case of BMGs, however, the reachable sets $\mathcal{R}$
do not coincide with the hierarchy $\mathcal{H}(T)$ of a tree explaining
$(G,\sigma)$. This is easily seen by considering
\begin{equation}
  \bigcup_{x\in V(G)} R(x) = \{ y\in V(G)\mid  N^-(y)\ne\emptyset \}
  = V\setminus V_{\text{source}}
\end{equation}
since a source vertex $y$ with $N^-(y)=\emptyset$ is never contained any
set $N(x)$ and thus also not in any reachable set $R(x)$.  As discussed 
in~\cite{Geiss:19a} for the special case of BMGs, it is not sufficient to
simply add $\{x\}$ to $R(x)$, however. Instead, larger sets are required to
handle source vertices as well as vertices that are indistinguishable in
terms of their in-neighborhood. Here, we consider a slightly modified
construction that also accommodates sinks. To this end, we define for any
digraph $G$ and every vertex $x\in V(G)$ the set
\begin{equation}
  Q(x)\coloneqq \left\{y\in V(G) \mid
  N^-(y)= N^-(x) \text{ and }
  \emptyset \ne N(y)\subseteq N(x)\right\}.
\end{equation}
This definition differs from the specification in~\cite{Geiss:19a} in two
aspects: (i) it adds the conditions ``$\emptyset\ne N(y)$'', which is
always satisfied in sink-free graphs and thus in BMGs. (ii) Each vertex is
considered separately here instead of being aggregated into so-called
thinness classes. We collect several simple properties of $Q(x)$ in
Lemma~\ref{lem:L10} below. The straightforward arguments are essentially
the same as in~\cite[Lemma~10]{Geiss:19a}.

\begin{lemma}
  \label{lem:L10}
  Let $(G,\sigma)$ be a properly 2-colored digraph.
  \begin{itemize}\setlength{\itemsep}{0pt}
    \item[(o)]   If $x$ is a sink, i.e., $N(x)=\emptyset$, then
    $Q(x)=\emptyset$; otherwise $N(x)\neq \emptyset$ implies $x\in 
    Q(x)$.
    \item[(i)]   $y\in Q(x)$ implies $\sigma(x)=\sigma(y)$.
    \item[(ii)]  $y\in Q(x)$ implies $Q(y)\subseteq Q(x)$.
    \item[(iii)] $y\in Q(x)$ implies $N(y)\subseteq N(x)$.
    \item[(iv)]  If $x\notin N(y)$, then $Q(x)\cap N(y)=\emptyset$.
    \item[(v)]   If $x\notin N(N(y))$, then $Q(x)\cap N(N(y))=\emptyset$.
    \item[(vi)] $N(x)\cap N(y)=\emptyset$ implies $Q(x)\cap Q(y)=\emptyset$.
  \end{itemize}
\end{lemma}
\begin{proof}
  (o) If $x$ is a sink, the condition becomes
  $\emptyset\ne N(y)\subseteq\emptyset$ and thus no such $y$
  exists. Otherwise, both conditions are trivially true for $y=x$.\\
  (i) Since $y\in Q(x)$ implies that $x$ and $y$ share at least one
  out-neighbor in the bipartite graph $G$, $x$ and $y$ are in the same
  color class.\\
  (ii) Let $y\in Q(x)$ and $z\in Q(y)$. Then by definition of $Q$,
  $ N^-(z)= N^-(y)= N^-(x)$ and
  $\emptyset\ne N(z)\subseteq N(y)\subseteq N(x)$, which implies
  $z\in Q(x)$ for all $z\in Q(y)$ and thus $Q(y)\subseteq Q(x)$.\\
  (iii) follows immediately from the definition of $Q$.\\
  (iv) Suppose $x\notin N(y)$ but there is $z\in Q(x)\cap N(y)$. Then
  $y\in N^-(z)= N^-(x)$ and thus $x\in N(y)$; a
  contradiction.\\
  (v) Suppose $x\notin N(N(y))$ but there is $z\in Q(x)\cap N(N(y))$.
  Thus there is $w\in N(y)$ such that $w\in
  N^-(z)= N^-(x)$ and therefore $x\in N(N(y))$; a
  contradiction.\\
  (vi) Let $N(x)\cap N(y)=\emptyset$ and assume, for contradiction, that  
  $Q(x)\cap Q(y)\neq \emptyset$. 
  Thus, there is a $z\in Q(x)\cap Q(y)$. By definition of $Q$, $\emptyset \neq 
  N(z) \subseteq N(x),N(y)$ and thus  $N(x)\cap N(y)\neq\emptyset$; a 
  contradiction.
\end{proof}

With the help of $Q$, we are now in the position to define the extended
reachable set for any digraph $G$ and every vertex $x\in V(G)$ as follows:
\begin{equation}
  R'(x) \coloneqq R(x)\cup Q(x)
\end{equation}
by analogy to the construction of 2-BMGs. For BMGs, the $R'(x)$ reduces to
the corresponding sets in~\cite{Geiss:19a} since BMGs have no sinks.  Note,
if $x$ is a sink, i.e., $N(x)=\emptyset$, then also $N(N(x))=\emptyset$ and
thus $R'(x) = \emptyset$. Conversely, $N(x)\subseteq R'(x) = \emptyset$
implies that $x$ is a sink. Therefore, we have
\begin{fact}\label{fact:R'empty}
  $R'(x) = \emptyset$ if and only if $x$ is a sink.
\end{fact}

\begin{fact}\label{fact:R'distinctColor}
  For a properly 2-colored digraph $(G,\sigma)$, we have $y\in N(x)$ if and
  only if $y\in R'(x)$ and $\sigma(x)\neq \sigma(y)$.
\end{fact}
\begin{proof}
  While the \emph{only-if}-direction is trivial, the \emph{if}-direction
  follows from the fact that all vertices in $N(N(x))$ and $Q(x)$ must have
  the same color as $x$.
\end{proof}

\par\noindent
The following result mirrors~\cite[Lemma~11]{Geiss:19a}.
\begin{lemma}
  \label{lem:ext-R-hierarchy}
  If $(G,\sigma)$ is properly two-colored graph satisfying \AX{(N1)},
  \AX{(N2)}, and \AX{(N3)}, then
  $\mathcal{R}'\coloneqq \{R'(x) \mid x\in V(G)\} \setminus \{\emptyset\}$
  forms a hierarchy-like set system.
\end{lemma}
\begin{proof}
  Let $R'(x), R'(y)\in \mathcal{R}'$ for two distinct vertices
  $x,y\in V(G)$.  By definition of $\mathcal{R}'$, neither of $R'(x)$ and
  $R'(y)$ is empty.  This and Observation~\ref{fact:R'empty} implies that 
  neither $x$ nor $y$ is a sink.
  
  We first show that $y\in R'(x)$ implies $R'(y)\subseteq R'(x)$. Let
  $y\in R'(x)$.  If $y\in Q(x)$, then Lemma~\ref{lem:L10} yields
  $Q(y)\subseteq Q(x)$ and $N(y)\subseteq N(x)$, and thus also
  $R(y)\subseteq R(x)$ and finally $R'(y)\subseteq R'(x)$. If $y\in R(x)$
  then $R(y)\subseteq R(x)\subseteq R'(x)$ since $y\in R(x)$ implies that
  everything that is reachable from $y$ is also reachable from $x$. By
  definition of $Q$, $N^-(z)= N^-(y)$ for all $z\in Q(y)$.  Thus
  $y\in R(x)=N(x)\cup N(N(x))$ implies
  $z\in N(x)\cup N(N(x))=R(x)\subseteq R'(x)$.  Hence, we have
  $Q(y)\subseteq R'(x)$ and, in summary,
  $R'(y)=R(y)\cup Q(y)\subseteq R'(x)$.  By analogous arguments,
  $x\in R'(y)$ implies $R'(x)\subseteq R'(y)$.
  
  Now suppose $y\notin R'(x)$ and $x\notin R'(y)$. We will show that this
  implies $R'(x)\cap R'(y)=\emptyset$. Since $x\notin R'(y)$, we have
  $x\notin N(y)$ and $x\notin N(N(y))$. This together with
  Lemma~\ref{lem:L10}(iv)-(v) implies $Q(x)\cap N(y)=\emptyset$ and
  $Q(x)\cap N(N(y))=\emptyset$, respectively. Thus
  $Q(x)\cap R(y)=\emptyset$. By similar arguments, we obtain
  $R(x)\cap Q(y)=\emptyset$ from $y\notin R'(x)$.  From $y\notin R'(x)$ and
  $x\notin R'(y)$, we have $y\notin N(x)$ and $x\notin N(y)$, which
  together with \AX{(N1)} implies
  $N(N(x)) \cap N(y) = N(x) \cap N(N(y))=\emptyset$.  Taken together, the
  latter arguments imply
  $R'(x)\cap R'(y)=(N(x)\cap N(y))\cup (N(N(x))\cap N(N(y))) \cup (Q(x)\cap
  Q(y))$. If $N(x)\cap N(y)=\emptyset$, Lemma~\ref{lem:N-empty-NN-empty}
  implies $N(N(x))\cap N(N(y))=\emptyset$ and Lemma~\ref{lem:L10}(vi)
  implies $Q(x)\cap Q(y)=\emptyset$. Therefore,
  $R'(x)\cap R'(y)=\emptyset$. Now assume $N(x)\cap N(y)\ne
  \emptyset$. From \AX{(N3)} the two cases $N(y)\subseteq N(x)$ or
  $N(x)\subseteq N(y)$ arise. In the former case, \AX{(N3')}, which is
  satisfied as a consequence of Lemma~\ref{lem:N1N2N3-old-eq-new}, together
  with $x\notin N(y)$, $y\notin N(x)$ and $N(x)\cap N(y)\ne \emptyset$
  implies $N^-(x)= N^-(y)$ and thus $Q(y)\subseteq Q(x)$.  Furthermore, we
  have $N(N(y))\subseteq N(N(x))$ by isotony of $N$, and thus
  $R'(y)\subseteq R'(x)$. Since $y\in R'(y)$, this contradicts
  $y\notin R'(x)$. In the latter case, analogous arguments yield
  $R'(x)\subseteq R'(y)$, contradicting $x\notin R'(y)$. Therefore, the
  case $N(x)\cap N(y)\neq \emptyset$ cannot occur and we indeed have
  $R'(x)\cap R'(y)=\emptyset$.
  
  In summary, $\mathcal{R}'$ forms a hierarchy-like set system.
\end{proof}

It should be stressed that we have made no assumptions about the connectedness
of $(G,\sigma)$, which has not been the case in the discussion 
in~\cite{Geiss:19a}.  The Hasse diagram $T(\mathcal{R}')$ therefore will in
general be a forest rather than a tree.  Still, the hierarchy-like set
system $\mathcal{R}'$ can easily be extended to a hierarchy by adding all
singletons as minimal elements and $V(G)$ as maximal element:
\begin{proposition}
  \label{prop:hierarchy}
  Let $(G,\sigma)$ be a properly 2-colored digraph satisfying \AX{(N1)},
  \AX{(N2)}, and \AX{(N3)}. Then
  \begin{equation}
    \mathcal{H}(G,\sigma) \coloneqq \mathcal{R}' \cup \left\{ \{x\}\mid x\in
    V(G)\right\} \cup \{ V(G) \}
  \end{equation}
  is a hierarchy on $V(G)$.
\end{proposition}

The following result parallels part of the proof of~\cite[Theorem~4]{Geiss:19a}.
\begin{lemma}
  \label{lem:n1n2n3-qBMG}
  Let $(G,\sigma)$ be a properly 2-colored digraph satisfying \AX{(N1)},
  \AX{(N2)}, and \AX{(N3)}, let $T$ be the tree with hierarchy 
  $ \mathcal{H}(T)  = \mathcal{H}(G,\sigma)$ and
  define for all $x\in V(G)$ and $r\ne\sigma(x)$ the truncation map by 
  $u(x,r)=\rho_T$ if
  $N(x)\ne\emptyset$ and $u(x,r)=x$ if $N(x)=\emptyset$. Then
  $(G,\sigma)=\qbmg(T,\sigma,u)$.
\end{lemma}
\begin{proof}
  By Proposition~\ref{prop:hierarchy}, $\mathcal{H}(G,\sigma)$ is a
  hierarchy on $V(G)$ and thus $L(T)=V(G)$.  In particular, $(G,\sigma)$
  and $\qbmg(T,\sigma,u)$ have the same vertex set.  Denote by
  $\tilde N(x)$ the out-neighbors of $x\in V(G)$ in $\qbmg(T,\sigma,u)$ and
  write $N(x)$ for the out-neighbors in $(G,\sigma)$.  We prove that
  $y\in \tilde N(x)$ if and only if $y\in N(x)$ for all $x,y\in V(G)$.  By
  Observation~\ref{fact:R'distinctColor}, $y\in N(x)$ is equivalent to
  $y\in R'(x)$ and $\sigma(y)\ne\sigma(x)$.
  
  First assume $y\in \tilde N(x)$.  By construction, we have
  $\sigma(y)\ne\sigma(x)$ and $N(x)\ne \emptyset$ since otherwise
  $u(x,\sigma(y))=x$ contradicts $y\in\tilde N(x)$. Assume, for
  contradiction, that $y\notin R'(x)$. Since $N(x)\ne\emptyset$, there is a
  vertex $y'\in N(x)\subseteq R'(x)$ such that $y'\ne y$ and
  $\sigma(y')=\sigma(y)$. Since $\mathcal{H}(G,\sigma)$ is a hierarchy, see
  Proposition~\ref{prop:hierarchy}, there is a unique inclusion-minimal set
  $R^*\in\mathcal{H}(G,\sigma)$ with $x,y \in R^*$.  However, $x\in
  R'(x)\cap R^*$ together with $y\notin R'(x)$ and the fact that $R'(x)$
  and $R^*$ are both sets in the hierarchy $\mathcal{H}(G,\sigma)$ implies
  that $R'(x)\subsetneq R^*$. Denoting by $v$ and $v^*$ the vertices of $T$
  satisfying $L(T(v))=R'(x)$ and $L(T(v^*))=R^*$, respectively, we have
  $v\prec_T v^*$. Since $R^*$ is inclusion-minimal with respect to
  the property $x,y \in R^*$, we have $\lca_{T}(x,y)=v^*$.  Similarly,
  $x,y'\in R'(x)$ implies that $\lca_T(x,y')\preceq_T v$. In summary, we
  obtain $\lca_T(x,y')\preceq_T v\prec_T v^*=\lca_{T}(x,y)$ and
  $\sigma(y)=\sigma(y')$. Thus $y$ cannot be a best match of $x$, which
  contradicts $y\in\tilde N(x)$. Therefore, we have $y\in R'(x)$ and thus
  $y\in N(x)$.
  
  For the converse, assume $y\in N(x)$, i.e., $y\in R'(x)$ and
  $\sigma(y)\ne\sigma(x)$.  In particular, therefore, the truncation map
  $u(x,\sigma(y))=\rho_T$ imposes no further constraint and
  $y\in \tilde{N}(x)$ if and only if $y$ is a best match of $x$.  Now
  assume, for contradiction, that $y$ is not a best match of $x$, i.e.,
  there is some vertex $y'\in V(G)$ of color $\sigma(y')=\sigma(y)$ such
  that $v\coloneqq\lca_{T}(x,y')\prec_T \lca_T(x,y)$.  Thus, we have
  $x,y'\in L(T(v))$ and $y\notin L(T(v))$.  In particular, by construction
  of $\mathcal{R}'$, there is a vertex $z\in V(G)$ such that
  $L(T(v))=R'(z)$.  Since $y\in R'(x)$ and $y\notin R'(z)$, we have
  $x\ne z$.  We have to consider three cases according to the constituents
  of $R'(z)=N(z)\cup N(N(z))\cup Q(z)$: If $x\in N(z)$ then
  $y\in N(N(z))\subseteq R'(z)$; a contradiction. If $x\in N(N(z))$, then
  there is $w\notin \{x,y,z\}$ such that $zw,wx,xy\in E(G)$ which, together
  with \AX{(N2)}, implies $y\in N(z)\subseteq R'(z)$; a contradiction.
  Finally, if $x\in Q(z)$, then by definition
  $y\in N(x)\subseteq N(z)\subseteq R'(z)$; a contradiction.  Since none of
  these cases is possible, we conclude that no vertex $y'$ as specified can
  exists and thus $y$ is a best match of $x$, i.e., $y\in \tilde{N}(x)$.
  
  Now, $y\in \tilde N(x)$ if and only if $y\in N(x)$ for all $x,y\in V(G)$
  immediately implies $\tilde N(x)=N(x)$ for all $x\in V(G)$, and thus the
  two 2-colored graphs coincide.
\end{proof}

Combining Lemma~\ref{lem:n1n2n3_quasi} and Lemma~\ref{lem:n1n2n3-qBMG} we
obtain the desired characterization of qBMGs:
\begin{theorem}\label{thm:char2qBMG}
  A properly 2-colored graph is a qBMG if and only if it satisfies
  \AX{(N1)}, \AX{(N2)}, and \AX{(N3)}.
\end{theorem}

The axioms \AX{(N1)}, \AX{(N2)}, and \AX{(N3)} are independent of the
coloring $\sigma$. It is of interest, therefore, to consider digraphs
satisfying these three conditions without considering a coloring. If $G$ is
bipartite, then a proper 2-coloring $\sigma$ exists that turns
$(G,\sigma)$ into a qBMG. This begs the question of whether the axioms
already imply that $G$ is bipartite.  We give an affirmative answer in the
following theorem.

\begin{theorem}
  Let $G$ be a digraph satisfying \AX{(N1)} and \AX{(N2)}. Then $G$ is
  bipartite.
  \label{thm:bipartite}
\end{theorem}
\begin{proof}
  A directed graph $G$ is a \emph{minimal counterexample} to the assertion
  of the theorem if it is minimal among the induced subgraphs of
  non-bipartite digraph satisfying \AX{(N1)} and \AX{(N2)}. $G$ is not
  bipartite if and only if its underlying undirected graph contains an
  \emph{odd cycle}. In the setting of directed graphs, an odd cycle is a
  subgraph of $G$ with vertex set $C=\{x_1,x_2,\dots,x_m\}\subseteq V(G)$
  of odd cardinality $m=|C|$ such that $x_{i-1} x_i \in E(G)$ or
  $x_i x_{i-1} \in E(G)$ for $1\le i\le \ell$, where $x_0=x_m$.  Two
  vertices $x_i$ and $x_j$ are consecutive if $j=i\pm1$.  Since \AX{(N1)}
  and \AX{(N2)} are hereditary, in particular $G[C]$ satisfies \AX{(N1)}
  and \AX{(N2)}. A minimal counterexample $G$, therefore, contains only the
  vertices of the odd cycle but possibly additional edges.
  
  If $G$ contains a triangle $\{x,y,z\}$ then, w.l.o.g., there is a
  directed path $xy,yz$. Then $xz\in E(G)$, since otherwise $zx\in E(G)$
  must complete the triangle in which case $x\in N(N(N(x)))$, contradicting
  \AX{(N2)}. However, $xz\in E(G)$ implies $z\in N(x)\cap N(N(x))$, which
  contradicts \AX{(N1)}. Thus $G$ is triangle-free and $|C|\ge 5$. In
  particular, therefore, there are two distinct, non-consecutive vertex $x$
  and $y$.  Denote by $C_{xy}\subsetneq C$ and $C_{yx}\subsetneq C$ the
  vertex sets of the distinct ``underlying undirected'' paths connecting
  $x$ and $y$ in $C$. A pair of vertices $\{x_i,x_j\}$ is a \emph{chord} in
  $C$ if $x_i x_j\in E(G)$ and $x_i$ and $x_j$ are non-consecutive along
  $C$, i.e., $|i-j|\ne 1$ for $1\le i\le m$ with $x_l=x_0$. If $\{x,y\}$ is
  a chord in $C$, then both $G[C_{xy}]$ and $G[C_{yx}]$ are cycles and
  $|C_{xy}|,|C_{yx}|\ge 3$ since $x$ and $y$ are non-consecutive along
  $C$. Since $C_{xy}\cap C_{yx}=\{x,y\}$ we have $|C_{xy}|+|C_{yx}|=|C|+2$
  and thus $|C_{xy}|,|C_{yx}|\le |C|-1$. Therefore, $G[C]$ is composed of
  two shorter cycles, of which one is even, and the other is odd;  say
  $C_{xy}$ is odd. This odd cycle is shorter than $C$. Again, since
  \AX{(N1)} and \AX{(N2)} are hereditary, in particular $G[C_{xy}]$
  satisfies \AX{(N1)} and \AX{(N2)}, and thus $G[C_{xy}]$ is a
  counterexample contradicting minimality of $G[C]$.  Hence, a minimal
  counterexample $G$ is isomorphic to a chordless odd cycle $C$ that is not
  a triangle and thus comprises $|C|\ge 5$ vertices.
  
  Assume that the minimal counterexample $G$ contains a path with three
  consecutive edges with the same orientation, $x\rightarrow a \rightarrow
  b \rightarrow y$. Then $y\in N(N(N(x)))$ and \AX{(N2)} implies $y\in
  N(x)$.  If $x\ne y$, this implies that $G$ is either an even cycle or $G$
  has a chord $\{x,y\}$; a contradiction. If, on the other hand, $x=y$, we
  obtain a contradiction to $G$ being a simple graph.  Thus $G$ may contain
  at most two consecutive edges with the same orientation. If two such
  edges exist, then the next edge must have the opposite
  orientation. More precisely, there exists a (not necessarily induced)
  subgraph $x\rightarrow a \rightarrow z\leftarrow y$ in $G$. Hence, $z\in
  N(N(x))\cap N(y)$, and thus \AX{(N1)} implies $x\in N(y)$ or $y\in
  N(x)$. Therefore, $\{x,y\}$ is a chord whenever $|V(G)|\ge 5$, a
  contradiction to the observation above that a minimum counterexample is
  isomorphic to a chordless cycle $C$ with $|C|\ge 5$. Therefore, a minimal
  counterexample $G[C]$ cannot contain a pair of consecutive edges with the
  same orientation. Hence, any two consecutive edges must have alternating
  orientations. This implies that $|C|$ is even.  Therefore, no minimal
  counterexample exists.
\end{proof}

\begin{corollary}
  Let $G$ be a graph satisfying \AX{(N1)}, \AX{(N2)}, and \AX{(N3)}. Then
  $(G,\sigma)$ is a qBMG for every proper 2-coloring $\sigma$ of $G$.
  \label{cor:twocolor}
\end{corollary}
\begin{proof}
  Consider the decomposition $G=\bigcupdot_i G_i$ of $G$ into its connected
  components. Theorem~\ref{thm:bipartite} implies that $G$ and thus each of
  its connected components $G_i$ are bipartite. Let $\sigma_i$ be a
  2-coloring of $G_i$. Furthermore, by heredity, the $G_i$ satisfies
  \AX{(N1)}, \AX{(N2)}, and \AX{(N3)}. This together with
  Theorem~\ref{thm:char2qBMG} implies that $(G_i,\sigma_i)$ is a 2-qBMG and
  thus, has an explanation $G_i(T_i,\sigma_i,u_i)$. Now exchange color $r$
  and $s$ in $G_i$, and set $u'_i(x,r)\coloneqq u_i(x,s)$ and
  $u'_i(x,s)\coloneqq u_i(x,r)$ for all $x\in V(G)$. Clearly,
  $G_i(T_i,\sigma'_i,u'_i)$ is an explanation for $(G_i,\sigma'_i)$. As a
  consequence of Observation~\ref{fact:components} there is an explaining
  tree $(T,\sigma,u)$ for every proper 2-coloring of $G$.
\end{proof}

Recall that color-sink-freeness of properly 2-colored digraphs is
equivalent to claiming that they are sink-free.  Another way of expressing
that 2-BMGs are the (color-)sink-free 2-qBMGs is to say that 2-BMGs
are the 2-qBMGs that have an explanation $(T,\sigma,u)$ with
$u(x,s)=\rho_T$ for all $x\in L$ and $s\ne\sigma(x)$, see
Theorem~\ref{thm:NewChar}. This begs the question of whether there are
interesting subclasses of 2-qBMGs that are more general than
2-BMGs.  In practical applications, it is plausible to assume that
sequence similarities can reliably identify homologs, at least within some
range of evolutionary divergence.  This suggests restricting the
truncation map $u$ to be bounded away from the leaves.  In the following,
we will investigate in which cases a 2-qBMG can be explained by a
leaf-colored tree $(T,\sigma,u)$ with truncation map $u$ that satisfies
\begin{itemize}
  \item[\AX{(M)}] $z\prec_T u(z,r)$ for $r\ne\sigma(z)$ and all $z\in V(G)$.
\end{itemize}
In other words, a truncation map $u$ satisfies \AX{(M)} if and only if the
statement ``$u(x,s)=x$ if and only if $s=\sigma(x)$'' is satisfied.  We start 
with the
following simple technical result.
\begin{lemma}
  \label{lem:monochrome-subtree}
  Let $(G,\sigma)$ be a qBMG explained by $(T,\sigma,u)$ and $v\in V(T)$
  such that $|\sigma(L(T(v)))|=1$. Then $N^-(x)=N^-(y)$ holds for all
  $x,y\in L(T(v))$.
\end{lemma}
\begin{proof}
  Let $x,y\in L(T(v))$. Since $|\sigma(L(T(v)))|=1$, we have
  $\sigma(x)=\sigma(y)$.  Assume there is a vertex $q\in N^-(x)$. Then, we
  have $\sigma(x)\ne \sigma(q)$ and thus $q\notin L(T(v))$.  Hence,
  $v\prec_T \lca_{T}(x,q)=\lca_{T}(y,q)$. Set $w=\lca_{T}(y,q)$.  Now
  $q\in N^-(x)$ implies $w\preceq_T u(q,\sigma(x))$. Since
  $\sigma(x)=\sigma(y)$ we have $u(q,\sigma(x))=u(q,\sigma(y))$ and thus,
  $q\in N^-(y)$. By the same arguments, $q\in N^-(y)$ implies
  $q\in N^-(x)$.  Therefore, we obtain $N^-(x)=N^-(y)$.
\end{proof}

\begin{definition}
  A graph $(G,\sigma)$ satisfies condition \AX{(N4)} if, for all
  $x\in V(G)$ with $N(x)=\emptyset$, there is $y\in V(G)\setminus\{x\}$
  such that $\sigma(y)=\sigma(x)$, and $N^-(x)=N^-(y)$.
\end{definition}

\begin{proposition}
  \label{prop:N4-no-leaf-truncation}
  Let $(G,\sigma)$ be a 2-qBMG. Then the following two statements are
  equivalent.
  \begin{itemize}\setlength{\itemsep}{0pt}
    \item[(i)] $(G,\sigma)$ is explained by a tree $T$ and a truncation map
    $u$ on $T$ satisfying \AX{(M)}.
    \item[(ii)] $(G,\sigma)$ satisfies condition \AX{(N4)}.
  \end{itemize}
\end{proposition}
\begin{proof}
  Set $L=V(G)$ and let $\sigma:L\to S$ be a proper 2-coloring of
  $G$. Suppose first that $(G,\sigma)$ can be explained by $(T,\sigma,u)$
  satisfying \AX{(M)}. Define $x_u\coloneqq u(x,s)$ to be the truncation
  vertex of $x$ for the color $s\ne\sigma(x)$ in $T$, and consider the
  subtree $T(x_u)$.  By the definition of the truncation maps, we have
  $x\preceq_T x_u$ and thus, $x\in L(T(x_u))$.  Suppose
  $|\sigma(L(T(x_u)))| = 2$. Hence, there is at least one $z\in L(T(x_u))$
  of color $\sigma(z)\neq \sigma(x)$ such that $z$ is a best match of $x$
  and $\lca(x,z)\preceq x_u$. Hence, $z\in N(x)$; contradicting the
  assumption that $N(x)=\emptyset$. Hence, $|\sigma(L(T(u_x)))|=1$.  Since
  $x_u$ is not a leaf (because $(T,\sigma,u)$ satisfies \AX{(M)}) and since
  $T$ is phylogenetic, there is a leaf $y\in L(T(x_u))$ with $y\ne x$ and
  $|\sigma(L(T(u_x)))|=1$ implies $\sigma(x)=\sigma(y)$. Together with
  Lemma~\ref{lem:monochrome-subtree} this yields $N^-(x)=N^-(y)$. We
  emphasize that $N^-(x)=N^-(y)=\emptyset$ is still possible.
  
  For the converse, assume that the 2-qBMG $(G,\sigma)$ satisfies
  \AX{(N4)}. By assumption, $|\sigma(L)|=2$ and thus $(G,\sigma)$ contains
  at least one vertex for each of two distinct colors and at least two
  vertices.  As a consequence of the latter, every leaf of an
  explaining tree has a parent.  We will prove the statement by
  constructing a finite sequence of trees $(T_i,\sigma,u_i)$ with
  $1\leq i \leq k$, each of which explains $(G,\sigma)$ and such that such
  that $(T_i,\sigma,u_i)$ differs from its predecessor and successor, and
  $(T_k,\sigma,u_k)$ satisfies the desired property \AX{(M)}.
  
  Since $(G,\sigma)$ is a qBMG, there is always a tree explaining
  $(G,\sigma)$. Denote this tree by $(T_1,\sigma,u_1)$. If
  $(T_1,\sigma,u_1)$ satisfies \AX{(M)}, $k=1$ and the proposition
  follows. Otherwise, we construct $(T_2,\sigma,u_2)$ in the following way:
  Since $(T_1,\sigma,u_1)$ does not satisfy \AX{(M)}, there is at least one
  vertex $x\in L$ such that $x_{u_1}=x$ for the color $r\ne \sigma(x)$ and
  thus $N(x)=\emptyset$. Therefore, \AX{(N4)} ensures that there is a
  vertex $y\in L'\coloneqq L\setminus\{x\}$ such that $\sigma(y)=\sigma(x)$
  and $N^-(x)=N^-(y)$. Set $w=\parent_{T_1}(x)$.  We distinguish two cases:
  (a) $\sigma(L(T_1(w)))=\{\sigma(x)\}$ and (b)
  $\{\sigma(x)\}\subsetneq \sigma(L(T_1(w)))$.
  
  In Case~(a) set $T_2=T_1$, $x_{u_2}=w$ and $z_{u_2}=z_{u_1}$ for all
  $z\in V(G)\setminus\{x\}$, and put $(G_2, \sigma)=\qbmg(T_2,\sigma,u_2)$.
  We observe that $x$ is still a sink in $G_2$.  By construction, the
  out-neighbors of all $z\in V(G)\setminus\{x\}$ also have not changed.
  Therefore, $\qbmg(T_2,\sigma,u_2)=(G,\sigma)$.
  
  In Case~(b) $|L|\ge 3$ and thus $|L'|\ge 2$ because $G$ contains vertices
  $x$ and $y$ and at least one vertex of the opposite color.  In order to
  construct $(T_2,\sigma,u_2)$, we first restrict the tree $T_1$ to $L'$ say
  $T'$, and then obtain $T_2$ from $T'$ by splitting the edge
  $\parent_{T'}(y) y$, i.e., we replace the edge by a newly-created vertex
  $p$ together with the edges $\parent_{T'}(y) p$ and $py$, and attaching
  $x$ as the second child of $p$. Recall from the proof of
  Lemma~\ref{lem:delvert} that $w'\in V(T_1)\setminus V(T')=V(T_1)\setminus
  (V(T_2)\setminus\{x\})$ is only possible if either $w'=x$ or $w'=w$ and
  $\child_{T_1}(w)=\{x,v^*\}$.  Now take $z_{u_2}=z_{u_1}$ if $z_{u_1}\in
  V(T_2)$ (that is if either $z_{u_1}\ne w$ or $w$ has more than two
  children in $T_1$) and $z_{u_2}=v^*$ otherwise. Finally, set $x_u=p$.
  Our notation identifies corresponding vertices of $T_1$ and
  $T_2$. Moreover, observe that the ancestor order $\preceq_{T_2}$ is
  preserved with respect to $\preceq_{T_1}$ restricted to
  $V(T_2)\setminus \{p\}$. Now we show that $(T_2,\sigma,u_2)$ also explains
  $(G,\sigma)$. Denote by $(G',\sigma)$ the digraph explained by
  $(T_2,\sigma,u_2)$. Since $\sigma(x)=\sigma(y)$, we have by construction
  $|\sigma(L(T_2(p)))|=1$. Together with Lemma~\ref{lem:monochrome-subtree},
  this implies $N^-_{G'}(x)=N^-_{G'}(y)$. Moreover,
  $N_{G'}(x)=\emptyset=N_{G}(x)$ follows from $x_u=p$. Therefore, we may
  restrict ourselves to vertices $x',y'\in L'$ with
  $\sigma(x')\ne\sigma(y')$.  Observe that $\lca_{T_1}(x',y')\notin V(T_2)$
  implies that $\lca_{T_1}(x',y')=w$ and that $w$ has only two children
  $\{x,v^*\}$, which yields $x\in\{x',y'\}$; a contradiction.  Thus,
  $\lca_{T_1}(x',y')\in V(T_2)$.
  
  Now suppose that $y'\in N_{G}(x')$.  Therefore,
  $\lca_{T_1}(x',y')\preceq_{T_1} x'_{u_1}$.  If $x'_{u_1}\in V(T_2)$, we
  have $x'_{u_2}=x'_{u_1}$. Otherwise, $x'_{u_1}=w$ and
  $\child_{T_1}(w)=\{x,v^*\}$. This implies
  $\lca_{T_1}(x',y')\preceq_{T_1} v^* \prec_{T_1} w$ because
  $\lca_{T_1}(x',y')=w$ occurs only if $x\in\{x',y'\}$. Therefore,
  $\lca_{T_2}(x',y') \preceq_{T_2} x'_{u_2}$ in both cases. Now assume, for
  contradiction, that $y'\notin N_{G'}(x')$. Since
  $\lca_{T_2}(x',y') \preceq_{T_2} x'_{u_2}$, this implies that there is
  $y''\in L\cap N_{G'}(x')$ of color $\sigma(y'')=\sigma(y')$ such that
  $\lca_{T_2}(x',y'')\prec_{T_2}\lca_{T_2}(x',y')$. In this case, we must
  have $y''=x$; otherwise $\lca_{T_2}(x',y'')\in V(T_1)$ and
  $\lca_{T_2}(x',y')\in V(T_1)$ implies
  $\lca_{T_1}(x',y'')\prec_{T_1}\lca_{T_1}(x',y')$, contradicting
  $y'\in N_{G}(x')$. Now $x=y''\in N_{G'}(x')$ and
  $N^-_{G'}(x)=N^-_{G'}(y)$ imply $y\in N_{G'}(x')$ and thus
  $\lca_{T_2}(x',y)=\lca_{T_2}(x',x)\prec_{T_2}\lca_{T_2}(x',y')$. Since
  $x', y\in L'$, $\lca_{T_2}(x',y)\in V(T_1)$ and thus
  $\lca_{T_1}(x',y)\prec_{T_1}\lca_{T_1}(x',y')$, which together with
  $\sigma(y)=\sigma(x)=\sigma(y')$ contradicts $y'\in N_{G}(x')$.  Hence,
  $y'\in N_{G}(x')$.
  
  Now assume $y'\notin N_{G}(x')$. Thus either
  $x'_{u_1}\prec_{T_1} \lca_{T_1}(x',y')$ or there is a $y''\in L$ with
  $\sigma(y'')=\sigma(y')$ such that
  $\lca_{T_1}(x',y'')\prec_{T_1} \lca_{T_1}(x',y')$. Suppose
  $x'_{u_1}\prec_{T_1} \lca_{T_1}(x',y')$. Then either
  $x'_{u_1},\lca_{T_1}(x',y')\in V(T_2)$, or $x'_{u_1}\not\in V(T_2)$. \ In
  the first case, $x'_{u_2}=x'_{u_1}\prec_{T_2}\lca_{T_2}(x',y')$ implies
  $y'\notin N_{G'}(x')$.  Alternatively,
  $x'_{u_2}=v^*\prec_{T_1} w=x'_{u_1}\prec_{T_1} \lca_{T_1}(x',y')$.  Since
  $v^*, \lca_{T_1}(x',y')\in V(T_2)$, we obtain
  $x'_{u_2}\prec_{T_2}\lca_{T_2}(x',y')$, and thus, $y'\notin
  N_{G'}(x')$. Finally, assume there is $y''\in N_{G}(x')$ with
  $y''\neq y'$.  If $y''\ne x$, then
  $\lca_{T_1}(x',y''),\lca_{T_1}(x',y')\in V(T_2)$ implies
  $\lca_{T_2}(x',y'')\prec_{T_2} \lca_{T_2}(x',y')$, and thus
  $y'\notin N_{G'}(x')$. On the other hand, if $y''=x$, then
  $N^-_{G}(x)=N^-_{G}(y)$ implies $y\in N_{G}(x')$, and thus
  $\lca_{T_2}(x',y)=\lca_{T_2}(x',y'')$. Since
  $\lca_{T_2}(x',y)\in V(T_2)$, we obtain
  $\lca_{T_2}(x',y)=\lca_{T_2}(x',y'')\prec_{T_2} \lca_{T_2}(x',y')$ which,
  together with $\sigma(y)=\sigma(x)$, yields $y'\notin N_{G'}(x')$.
  
  In summary, we have shown that $(G,\sigma)=(G',\sigma)$, and hence
  $(T_2,\sigma,u_2)$ also explains $(G,\sigma)$.
  
  e show that, for every $z\in L$, $z_{u_2}$ cannot be a leaf whenever
  $z_{u_1}$ is not a leaf.  Observe that $z\prec_{T_1} z_{u_1}$ implies
  $z\ne x$. If $z_{u_1}\in V(T_2)$, then $z_{u_2}=z_{u_1}$ and $z_{u_2}$ is
  clearly still an inner vertex. On the other hand, $z_{u_1}\notin V(T_2)$
  ensures $z_{u_1}=w$ and $\child_{T_1}(w)=\{x,v^*\}$. Since $z\ne x$ we
  have $z\preceq_{T_1} v^*$. If $\sigma(z)=\sigma(x)$, then
  $\{\sigma(x)\}\neq \sigma(L(T_1(w)))$ implies that there exists
  $z'\in L(T_1(v^*))$ such that $\sigma(z)\ne\sigma(z')$ and thus
  $z\neq z'$. Hence, $v^*=z_{u_2}$ is not a leaf.  On the other hand, if
  $\sigma(z)\ne\sigma(x)$, then there is some $x''\preceq_{T_1} v^*$ with
  $\sigma(x'')=\sigma(x)$ implying that $v^*$ is an inner vertex or, if
  not, $x\in N_{G}(z)$ as $\lca_{T_1}(z,x)=w=x_{u_1}$. In the latter case,
  $N^-_{G}(x)=N^-_{G}(y)$ yields $y\in N_{G}(z)$, and thus
  $w=\lca_{T_1}(z,x)=\lca_{T_1}(z,y)$. However, $x\ne y$ and
  $w=\parent_{T_1}(x)$ imply that there is
  $v^{**}\in \child_{T_1}(w)\setminus\{x,v^*\}$; a contradiction.  Hence,
  the case $z_{u_1}\notin V(T_2)$ and $\sigma(z)\ne\sigma(x)$ cannot occur.
  
  In both Cases~(a) and~(b), therefore, $(T_2,\sigma,u_2)$ explains
  $(G,\sigma)$ and $z_{u_2}$ can only be a leaf if $z_{u_1}$ is a
  leaf. Furthermore, $x_{u_2}$ is not a leaf, while $x_{u_1}=x$, i.e., a
  leaf. Therefore we have decreased the number of truncation vertices that
  are leaves.  Since the number of leaves is finite, we can repeat this
  procedure and eventually obtain a tree $(T_k,\sigma,u_k)$ for which
  $z_{u_k}\notin L$ for all leaves $z$. Thus the tree $(T_k,\sigma,u_k)$
  satisfies property~\AX{(M)}.
\end{proof}

As a simple example, consider the digraph of two vertices with a single
edge. It satisfies \AX{(N1)}, \AX{(N2)}, and \AX{(N3)}, and thus is a qBMG,
but it violates condition \AX{(N4)} and thus has no explanation without a
truncation at a leaf. Proposition~\ref{prop:N4-no-leaf-truncation} implies
that, for a 2-qBMG $(G,\sigma)$ that satisfies \AX{(N4)}, we always find an
explaining leaf-colored tree with truncation map $(T,\sigma,u)$ that
satisfies \AX{(M)}. However, a 2-qBMG $(G,\sigma)$ satisfying \AX{(N4)} may
also have explanations that do not satisfy \AX{(N4)}, as shown by the
example in Figure~\ref{fig:N4-non-hereditary}.

\begin{figure}[tb]
  \centering
  \includegraphics[width=0.85\linewidth]{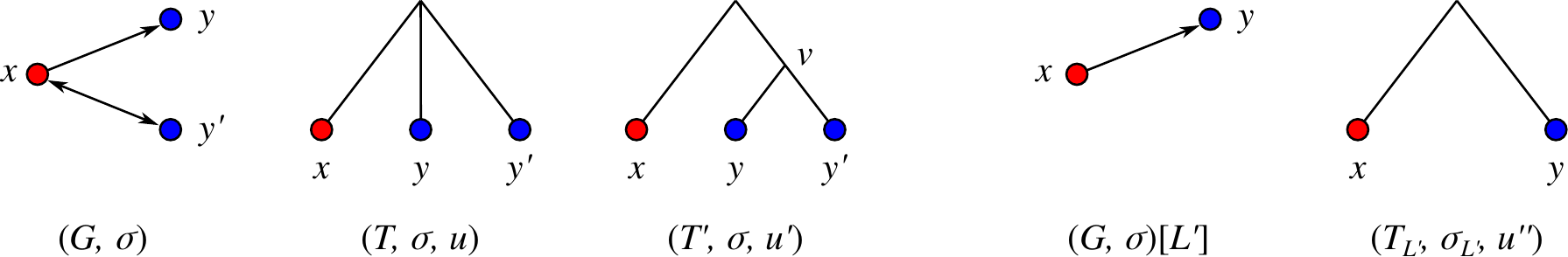}
  \caption{The qBMG $(G,\sigma)$ can be explained by both $(T,\sigma, u)$
    and $(T',\sigma,u')$ for suitable choices of $u$ and $u'$,
    respectively.  However, since the root is the only non-leaf vertex in
    $T$, there is no explanation $(T',\sigma,u')$ that satisfies
    \AX{(M)}. The tree $T'$ has an additional inner vertex $v$. Therefore,
    we can set $u'(y,\sigma(x))=v$.\hfill\break The induced subgraph
    $(G,\sigma)[L']$ with $L=\{x,y\}$ has a sink $y$ but no
    $y'\in L'\setminus\{x\}$ such that $\sigma(y)=\sigma(y')$, and
    $N^-(y)=N^-(y')$. Hence, Property \AX{(N4)} is not hereditary.}
  \label{fig:N4-non-hereditary}
\end{figure}

Moreover, like sink-freeness, Property \AX{(N4)} is not
hereditary. To see this, let $x$ be a sink and suppose that
$\hat x \in L'\coloneqq L\setminus \{x\}$ is the only vertex with
$\sigma(x)=\sigma(\hat x)$ and $N^-(\hat x) =N^-(x)$. Then
$(G, \sigma)[L']$ is still a qBMG Lemma~\ref{cor:hereditary}.
We have $\hat x\notin N^{-}(x')$ for any $x'\in L'$ of color $\sigma(x)$
because of $\sigma(x)=\sigma(\hat x)$, and thus $N^-(x')$ remains unchanged
by deleting $\hat x$.  Hence, $(G, \sigma)[L']$ does not satisfy
\AX{(N4)}.

\section{Concluding Remarks and Open Questions}

In this contribution, we investigated a generalization of best match graphs
of~\cite{Geiss:19a}, which we termed qBMGs.  In the two-colored case, which
in particular occurs in the form of the subgraph induced by two vertex
colors, qBMGs are characterized by three simple conditions, \AX{(N1)},
\AX{(N2)}, and \AX{(N3)}. The first two conditions are already sufficient
to ensure that the graph is bipartite. In the general case, qBMGs are
characterized by their induced subgraphs on three vertices with two colors,
which translate to sets of ``informative'' and ``forbidden'' triples. It is,
therefore, possible to recognize $\ell$-colored qBMGs in polynomial time. In
the positive case, an explaining tree and a corresponding truncation map
can also be constructed in polynomial time.

In contrast to BMGs, being a qBMG is a hereditary property. On the other
hand, BMGs have unique least resolved trees (LRTs), a property that is no
longer true for qBMGs in general.  It will be interesting to ask how
different alternative LRTs can become. Since they must display the
informative triples, one would expect that there is some common
``core''. The uniqueness of LRTs on BMGs suggests considering (maximal)
induced subgraphs that are BMGs.  Are their LRTs displayed by all LRTs of
the qBMG? Is there an efficient algorithm to find all maximal induced BMGs
in a qBMG?

From an application point of view, the truncation map $u(x,r)$ describes
the phylogenetic scope within which a homolog $y$ of the query gene $x$ can
be found. If the target genomes have similar size and organization, it
becomes a reasonable approximation to assume that phylogenetic scope
depends only on the dissimilarity between query and target gene but not
on the identity of the target genome.  In this case, the truncation map
becomes independent of the color, i.e., we have $u(x,r)=u(x)$ for all
$r\ne\sigma(x)$. Assuming further that all genes have similar size and
internal structure further restricts $u$ to ``cutting'' the tree at a
certain height. It remains an open question whether these
subclasses of qBMGs also have interesting mathematical properties.
In the special case of two colors, a simple characterization was obtained
for qBMGs that can be explained by truncation maps that exclude the leaves.
Quasi-best match graphs that can be explained by a tree $(T,\sigma,u)$
satisfying \AX{(M)} are also of interest for more than two colors.

In Corollary~\ref{cor:twocolor} we characterized the digraphs $G$ that admit a
2-coloring $\sigma$ such that $(G,\sigma)$ is a 2-qBMG. Naturally, one
might want to ask which directed graphs admit an $\ell$-coloring
such that $(G,\sigma)$ is an $\ell$-qBMG.  On the other extreme, setting
$S=V(G)$ and using coloring $\sigma(x)=x$, we can use the rooted star tree
with truncation map $u(x,\sigma(y))=\rho$ if $xy\in E(G)$ and 
$u(x,\sigma(y))=x$ if
$xy\notin E(G)$ as explanation for $(G,\sigma)$. Thus every directed graph
$G$ can be colored to be a $|V(G)|$-qBMG.

\begin{figure}[tb]
  \centering
  \includegraphics[width=0.75\linewidth]{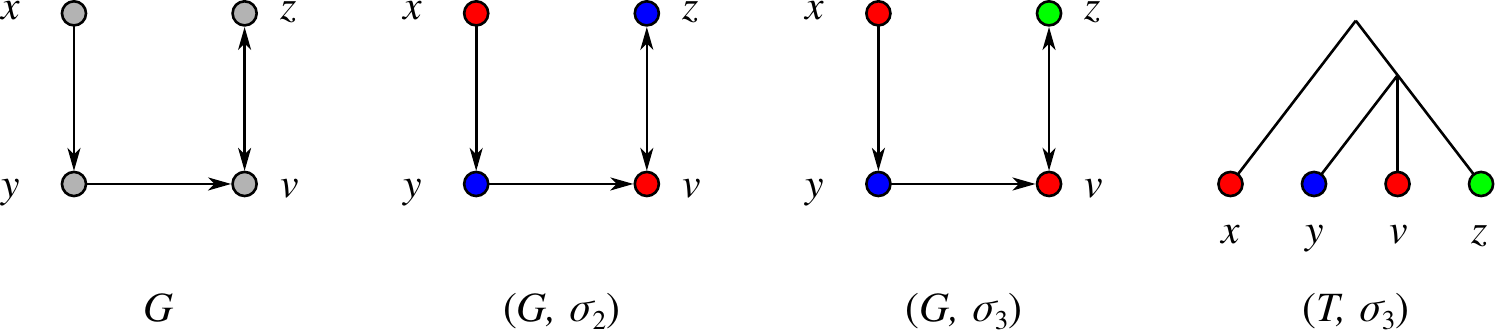}
  \caption{Example of a bipartite digraph $G$ with a qBMG-chromatic number
    of $3$.  The only proper 2-coloring (up to exchanging the colors) is
    $\sigma_2$.  However, $(G,\sigma_2)$ violates both \AX{(N1)} and
    \AX{(N2)} and thus, it is not a qBMG. For $\sigma_3$, we can easily find
    a truncation map $u$ such that $(G,\sigma_3)=\qbmg(T,\sigma_3,u)$ by
    setting 
    $u(v,\sigma_3(z))=u(y,\sigma_3(v))=u(x,\sigma_3(y))=u(z,\sigma_3(x))=
    \rho_T$, $u(x,\sigma_3(z))=x$, $u(v,\sigma_3(y))=v$, $u(z,\sigma_3(y))=z$ 
    and $u(y,\sigma_3(z))=y$.} 
  \label{fig:chromatic-number-3}
\end{figure}

This suggests to consider the \textsc{qBMG-coloring} problem: \emph{Given a
  digraph $G$ and an integer $1\le\ell\le |V(G)|$, is there a coloring
  $\sigma$ with $\sigma(V)=\ell$ such that $(G,\sigma)$ is an $\ell$-BMG?}
Proposition~\ref{prop:numcol}, furthermore, shows that every coloring
$\sigma$ of $G$ for which $(G,\sigma)$ is an $\ell$-qBMG with $\ell<|V(G)|$
can be transformed into an $(\ell+1)$ coloring $\sigma'$ by arbitrarily
splitting a color class resulting in an $(\ell+1)$-qBMG. Thus, for every
digraph, there is a minimum integer $\ell_{qBMG}$, the ``qBMG-chromatic
number'', such that an $\ell_{qBMG}$-coloring $\sigma$ exists for which
$(G,\sigma)$ is a qBMG. Figure~\ref{fig:chromatic-number-3} shows an
example of a bipartite graph $G$ that can be colored to be a 3-qBMG, while
there is no 2-coloring $\sigma$ such $(G,\sigma)$ is a 2-qBMG.

\bigskip\noindent
\textbf{Acknowledgments.}
This work was support in part by the German Research Foundation (DFG, STA
850/49-1).

\bibliography{qBMG-preprint2}

\end{document}